\documentclass[12pt]{amsart}
\usepackage{amsthm,amsmath,amssymb,amscd,amsfonts}

\begin{document}

\title[]
{Positivity of twisted relative pluricanonical bundles and their direct images}

\author{Mihai P\u aun and Shigeharu Takayama}

\address{Mihai P\u aun \endgraf
Korea Institute for Advanced Study, \endgraf
Seoul, 130-722
South KOREA}
\email{paun@kias.re.kr}

\address{Shigeharu Takayama \endgraf   
Graduate School of Mathematical Sciences,
The University of Tokyo \endgraf   
3-8-1 Komaba, Tokyo, 153-8914, 
Japan}
\email{taka@ms.u-tokyo.ac.jp}


\date{\today}
\maketitle
\tableofcontents
\baselineskip=16pt

\theoremstyle{plain}
  \newtheorem{thm}{Theorem}[subsection]
  \newtheorem{thm'}{Theorem}[section]
  \newtheorem{prop}[thm]{Proposition}
  \newtheorem{prop'}[thm']{Proposition}
  \newtheorem{lem}[thm]{Lemma}
  \newtheorem{cor}[thm]{Corollary}
\theoremstyle{definition}  
  \newtheorem{dfn}[thm]{Definition}
  \newtheorem{exmp}[thm]{Example}
  \newtheorem{prob}[thm]{Problem}
  \newtheorem{notation}[thm]{Notation}
  \newtheorem{quest}[thm]{Question}
  \newtheorem{rem}[thm]{Remark}
  \newtheorem*{acknowledgement}{Acknowledgement}
  \newtheorem{no}[thm]{}
  \newtheorem{setup}[thm]{Set up}
\numberwithin{equation}{subsection}

\newcommand{\BC}{{\mathbb{C}}}
\newcommand{\BN}{{\mathbb{N}}}
\newcommand{\BP}{{\mathbb{P}}}
\newcommand{\BQ}{{\mathbb{Q}}}
\newcommand{\BR}{{\mathbb{R}}}
\newcommand{\BZ}{{\mathbb{Z}}}
\newcommand{\BD}{{\mathbb{D}}}

\newcommand{\CE}{{\mathcal{E}}}
\newcommand{\CF}{{\mathcal{F}}}
\newcommand{\CH}{{\mathcal{H}}}
\newcommand{\CI}{{\mathcal{I}}}
\newcommand{\CJ}{{\mathcal{J}}}
\newcommand{\CO}{{\mathcal{O}}}
\newcommand{\CQ}{{\mathcal{Q}}}
\newcommand{\CS}{{\mathcal{S}}}
\newcommand{\cV}{{\mathcal{V}}}

\newcommand{\fm}{{\mathfrak {m}}}

\newcommand{\ga}{{\alpha}}
\newcommand{\gb}{{\beta}}
\newcommand{\gc}{{\gamma}}
\newcommand\ep{\varepsilon}
\newcommand{\Ga}{{\Gamma}}
\newcommand{\Del}{{\Delta}}
\newcommand{\Dl}{{\Delta}}
\newcommand{\del}{{\delta}}
\newcommand{\lam}{{\lambda}}
\newcommand{\Lam}{{\Lambda}}
\newcommand{\vph}{{\varphi}}
\newcommand{\sg}{{\sigma}}
\newcommand{\Th}{{\Theta}}
\newcommand\w{\omega}
\newcommand\Om{\Omega}

\newcommand\wtil{\widetilde}
\newcommand\what{\widehat}

\newcommand\sm{\setminus}
\newcommand\ol{\overline}
\newcommand\ot{\otimes}
\newcommand\wed{\wedge}

\newcommand\sumn{\sum\nolimits}

\newcommand\lra{\longrightarrow}
\newcommand\da{\downarrow}

\newcommand\ai{\sqrt{-1}}
\newcommand\rd{{\partial}}
\newcommand\rdb{{\overline{\partial}}}

\newcommand\Amp{\mbox{{\rm Amp}}\, }
\newcommand\codim{\text{{\rm codim}}\, } 
\newcommand\NAmp{\mbox{{\rm NAmp}}\, }
\newcommand\Proj{\mbox{{\rm Proj}}\, }
\newcommand\Vol{\mbox{{\rm Vol}}\, }
\newcommand\rank{\mbox{{\rm rank}}\, }
\newcommand\Reg{\mbox{{\rm Reg}}\, }
\newcommand\Sing{\mbox{{\rm Sing}}\, }
\newcommand\supp{\mbox{{\rm supp}}\, }
\newcommand\Supp{\mbox{{\rm Supp}}\, }
\newcommand\vol{\mbox{{\rm vol}}}


Abstract.
Our main goal in this article is to establish a  quantitative version of the positivity properties of twisted relative pluricanonical bundles and their direct images. 
Some of the important technical points of our proof are an $L^{2/m}$-extension theorem of Ohsawa-Takegoshi type which is derived from the original result by a simple fixed point  method, and the notion of ``singular Hermitian metric" on vector bundles, together with an appropriate definition of positivity of the associated curvature.
Part of this
article is based on the joint work of the first named author with Bo Berndtsson \cite[Part A]{BP10}, and it can be seen as an expanded and updated version of it.

\

\newpage

\section{Introduction}

We continue to investigate here the (quantitative version of the) positivity properties of twisted relative pluricanonical bundles and their direct images. A particular case of the main result we establish in this paper is as follows.

\begin{thm'}\label{MT}
Let $f:X\to Y$ be a projective surjective morphism of complex manifolds with connected fibers, and let $Y_{0}\subset Y$ be the maximum Zariski open subset where $f$ is smooth.
Let $L$ be a holomorphic line bundle on $X$ with a continuous Hermitian metric $h$ with semi-positive curvature current.
Suppose that the direct image sheaf $f_{\star}(mK_{X/Y}+L)$ is non-zero for a positive integer $m$.
Then

(1) the line bundle $mK_{X/Y}+L$ admits a singular Hermitian metric $h^{(m)}=B_{m}^{-1}$ with semi-positive curvature, and on every fiber $X_{y}$ over $y \in Y_{0}$, the restricted dual $B_{m}|_{X_y}$ is the twisted $m$-th Bergman kernel metric $B_{m,y}$ of $-(mK_{X_{y}}+L|_{X_{y}})$.
In particular, if $X$ is compact, $mK_{X/Y}+L$ is pseudo-effective.

(2) The torsion free sheaf $f_{\star}(mK_{X/Y}+L)$ admits a singular Hermitian metric $g_{m}$ with Griffiths semi-positive curvature, and $g_{m}$ is continuous on $Y_{0}$ and $0<\det g_{m} <\infty$ on $Y_{0}$, moreover $g_{m}$ on each fiber $f_{\star}(mK_{X/Y}+L)_{y}$, which is in fact $H^{0}(X_{y}, mK_{X_{y}}+L|_{X_{y}})$, with $y\in Y_{0}$ is the so-called $m$-th Narashimhan-Simha Hermitian form.
In particular, if $Y$ is projective, $f_{\star}(mK_{X/Y}+L)$ is weakly positive at every $y\in Y_{0}$ in the sense of Nakayama \cite[V.3.20]{Nbook}.
\end{thm'}

We also deal with the case when the metric $h$ is singular.
In such a case, the corresponding statements become more complicated naturally (see \ref{Grp} for example). 
Let us add some explanations on the statement being $L=\CO_{X}$ for simplicity.
The weak positivity $f_{\star}(mK_{X/Y})$ at  every point $y \in Y_{m}$ for some Zariski open $Y_{m}\subset Y$ is known classically if $Y$ is projective, by the research of Kawamata \cite{Ka82}, Viehweg \cite{Vi1} concerning the Iitaka conjecture, building on earlier works by Griffiths \cite{Gr} and Fujita \cite{Ft}.
As a corollary of it, the pseudo-effectivity of $mK_{X/Y}$ follows. Actually this field has generated many important works, including the twisted case by a semi-ample line bundle $L$; we will only mention here a few of them 
\cite{B}, \cite{Cam04}, \cite{Fn}, \cite{Ho}, \cite{Ka98}, \cite{Ka02}, \cite{Ka09}, \cite{Ko86}, \cite{Ko07}, \cite{MT08}, \cite{MT09}, \cite{Ny}, \cite{Vbook} (and we apologize to the authors we omit ...).

Our main contributions in the statement above are to provide a quantitative counterpart of (1) and (2) above, and to describe the open subset $Y_{m}\subset Y$ where these properties behave nicely.
By this we mean that the semi-positivity property of $mK_{X/Y}$, as well as the (pointwise) weak positivity property of $f_{\star}(mK_{X/Y})$ can be obtained as a consequence of the existence of the natural Bergman and Narashimhan-Simha type  metrics we construct on these sheaves. 
We would also like to stress that the Griffiths semi-positivity of vector bundles or more generally torsion free sheaves is in practice stronger than the weak positivity in algebraic geometry (recall for example, a Griffiths positive vector bundle over a projective manifold is ample, but the converse is not known; or a nef line bundle is not necessarily semi-positive), and that Viehweg's weak positivity does not imply the curvature semi-positivity of the metric $h^{(m)}$ of $mK_{X/Y}$ nor $g_{m}$ on $f_{\star}(mK_{X/Y})$.
The explicit fiberwise description of $B_{m}$ is crucial in the process of construction of the metric $g_{m}$ on $f_{\star}(mK_{X/Y})$.

The metrics $h^{(m)}$ and $g_{m}$ appeared in Theorem \ref{MT} are both canonically defined as we shall explain next (see \S \ref{abs} in general, also \cite{NS}, \cite{Ka82}, \cite{Ts07}). 
These metrics are defined in a way similar to the Bergman kernel metric, but using $L^{p}$-spaces instead of $L^{2}$-spaces, with $p=2/m$.
To explain this, we let here $X$ be a compact complex manifold with a holomorphic  Hermitian line bundle $(L,h)$.
Then the $m$-th Bergman kernel metric $B_m$ on the dual $(mK_{X}+L)^{\star}$ is given, at every point $x\in X$ and for every vector $\xi_{x} \in (mK_{X}+L)^{\star}_{x}$, by
$$
	|\xi_{x}| = \sup |\xi_{x} \cdot  u(x) |,
$$ 
where the supremum being taken over all $u \in H^{0}(X,mK_{X}+L)$ with 
$$
	\|u\|_{m} := \left(\int_{X} (c u \wed \ol u h)^{1/m} \right)^{m/2} \le 1
$$ 
with a constant $c\in \BC$ which makes the pairing $c u \wed \ol u h$ positive definite.
Then the singular Hermitian metric $h^{(m)}$ on $mK_{X}+L$ is defined by the dual metric of the $m$-th Bergman kernel metric, i.e., $h^{(m)}=B_{m}^{-1}$.
Using $B_{m}$, we also define a Hermitian form $g_{m}$ on $H^{0}(X, mK_{X}+L)$ given by 
$$
	g_m(u,v)
= \int_{X} (-1)^{n^{2}/2} u\wed \ol v \, (B_m^{-1})^{(m-1)/m} h^{1/m},
$$ 
where $n=\dim X$.
When $m=1$, $g_{1}$ is nothing but the canonical $L^{2}$-pairing $\int_{X}(-1)^{n^{2}/2}u \wed \ol v h$ on $H^{0}(X, K_{X}+L)$.
Also in general $m\ge 1$, letting $h_{m-1}=(B_m^{-1})^{(m-1)/m} h^{1/m}$ be a singular Hermitian metric on $L_{m-1}=(m-1)K_{X}+L$, $g_{m}$ can be seen as the canonical $L^{2}$-pairing on $H^{0}(X, K_{X}+L_{m-1})$.
Theorem \ref{MT} asserts that the fiberwise metrics $h^{(m)}_{y}$, the pairings $g_{m,y}$ respectively, glues together nicely over $Y_{0}$ and has curvature positivity, and then extends over $Y$ across the singularities of $f$ keeping the curvature positivity.

\medskip

The main differences between Theorem \ref{MT} in its general form and the corresponding previous result \cite[4.2]{BPDuke} are as follows: (i) the morphism $f$ can be singular, (ii) the direct image sheaf can merely be torsion free, and (iii) the subset of $Y$ where the metrics $h^{(m)}$ and $g_{m}$ can be described explicitly 
is given in a very accurate manner.
To obtain these properties, there are some new ingredients which are introduced in this paper.
One is the $L^{2/m}$-version of Ohsawa-Takegoshi's extension theorem for Theorem \ref{MT}(1), another is a notion of singular Hermitian metrics on vector bundles and torsion free sheaves and its curvature positivity/negativity for Theorem \ref{MT}(2).

We recall Ohsawa-Takegoshi's $L^2$-extension theorem \cite{OT} (cf.\ \cite[12.9]{Dnote}) in a very basic set-up.
Let $\Omega\subset \BC^n$ be a ball of radius $r$ and let $\sg:\Omega\to \BC$ be a holomorphic function on $\Om$; $\sg\in \CO(\Om)$, such that $\sup_\Omega |\sg| \leq 1$, and that the gradient $\partial \sg$ of $\sg$ is nowhere zero on the set $V:=  \{\sg=0\}$. 
In particular $V$ is a smooth complex hypersurface in $\Om$.
We denote by $\varphi$ a plurisubharmonic function on $\Om$ such that its restriction to $V$ is well-defined (i.e., $\varphi|_{V}\not\equiv -\infty$). 
Then Ohsawa-Takegoshi's theorem states that for any $f \in \CO(V)$,
there exists $F \in \CO(\Om)$ such that $F|_{V}=f$ and 
$$
	\int_{\Omega} |F|^{2} \exp (-\varphi) d\lambda 
	\leq C_0 \int_{V} |f|^{2} \exp (-\varphi){{d\lambda_V}\over {|\partial \sg|^2}}, 
$$
where $C_0$ is an absolute constant, and where $d\lam$ and $d\lam_{V}$ are the volume form induced from the Euclidean metric on $\BC^n$. 
(The right hand side can formally be $+\infty$.)
Recently, the optimal bound of $C_{0}$ is obtained by B\l ocki \cite{blocki} and Guan-Zhou \cite{GZ}.
Then our generalization in the set-up above is the following (see \ref{SOT}).

\begin{prop'}\label{mOT1}
Let $m\ge 1$ be a real number.
Then for any $f \in \CO(V)$, there exists $F\in {\CO}(\Omega)$ such that $F|_{V}= f$ and satisfying the $L^{2/m}$-bound 
$$
	\int_{\Omega} |F|^{2/m}\exp (-\varphi)d\lambda 
	\leq C_0\int_{V} |f|^{2/m}\exp (-\varphi){{d\lambda_V}\over {|\partial \sg|^2}}, 
$$
where $C_0$ is the same constant as in the Ohsawa-Takegoshi theorem.
\end{prop'}

Once this result is established, the proof of Theorem \ref{MT}(1) runs as follows.
For any $y\in Y_{0}$, all the elements in $H^{0}(X_{y}, mK_{X_{y}}+L|_{X_{y}})$ extend locally near $y$, thanks to Siu's invariance plurigenera \cite{Siu} (see also \cite{Paun}). 
Thus over $Y_0$ (i.e., on $f^{-1}(Y_{0}) \subset X$), we can apply \cite[4.2]{BPDuke} and therefore the logarithm of the $m$-th Bergman kernel metric has a psh variation. 
We then use the  $L^{2/m}$-extension result to estimate our metric from above by a uniform constant. 
Standard results of pluri-potential theory then gives that the metric extends, across the singular fibers of $f$, to a semi-positively curved metric on all of $X$.

\medskip

Aiming to understand the metric properties of direct image sheaves $f_{\star}(mK_{X/Y}+L)$ in Theorem \ref{MT}(2), we explore a few aspects of the notion of \emph{singular Hermitian metric} on vector bundles and more generally on torsion free sheaves, in \S \ref{SsHm}. 
This notion is well understood and extremely useful in rank 1, i.e.\ for line bundles. 
Motivated by the study of direct image sheaves, one would like to develop a similar theory in the context of vector bundles $E$ of $\rank (E) \ge 2$.
At first glance it may look surprising that there are not so many works dedicated to this subject --with the notable exceptions of \cite{dC}, \cite[\S 3]{BPDuke}, \cite{Raufi1}, \cite{Raufi2}.
The observation which is central in the present article is that ``in practice" what we actually need is to define a notion of negativity or positivity of a (smooth) Hermitian vector bundle $(E, h_E)$ in order to derive some consequences of this property. 
The positivity in the sense of Griffiths has a very important property: the bundle $(E, h_E)$ is Griffiths positive if and only if the dual Hermitian bundle $(E^\star, h_{E^\star})$ is Griffiths negative (unfortunately, the Nakano positivity fails to satisfy this property).
It is then easy to see that the Griffiths semi-positivity of $(E, h_E)$ can be characterized as follows. 
Let $u \in H^{0}(\Om, E^{\star})$ be any non-zero section, where $\Omega\subset X$ is a coordinate open subset. 
The function $\displaystyle \log |u|_{h_{E}^\star}$ is psh on $\Omega$  if and only if $(E, h_E)$ is Griffiths semi-positive. 
Thus, we are able to formulate this notion without an explicit mention of the curvature tensor associated to $(E, h_E)$.

The above property of (smooth) Hermitian vector bundles is converted into the definition of the Griffiths semi-positivity in the context of singular Hermitian vector bundles, see \cite[\S 3]{BPDuke}. 
By definition, a singular Hermitian metric $h_E$ on the trivial vector bundle $E$ of rank $r$ defined on an open set $\Omega\subset \BC^n$ is simply a measurable map defined on $\Omega$ with values in the space of positive semi-definite $r\times r$-Hermitian matrices, such that $0<\det (h_E)<\infty$ almost everywhere. This later condition is imposed so as to avoid the metric $h_E$ to be everywhere degenerate e.g.\ at each point of $\Omega$ along some directions in $E$.  
If $E$ is non-trivial, a singular Hermitian metric $h_{E}$ on $E$ corresponds to a collection of singular Hermitian metrics on each local trivialization of $E$, verifying the usual glueing condition.
By the discussion above, we can define a notion of Griffiths semi-positivity for $(E, h_E)$.
The corresponding Nakano semi-positivity of a singular Hermitian metric does not seem to be well-understood yet; nevertheless we refer to \cite{Raufi1}, \cite{Raufi2} for the state of arts in this direction.

\noindent 
After treating the basic properties of singular Hermitian metrics (on torsion free sheaves, for example), we obtain the following general result \ref{Gp imply wp2} which is connected with algebraic geometry.
In fact we establish a more general Finsler metric version \ref{bigness}, which gives criteria for the pointwise bigness and the pointwise weak positivity of a torsion free sheaf.

\begin{thm'}
Let $Y$ be a smooth projective variety, and let $\CE$ be a torsion free sheaf on $Y$.
Suppose that $\CE$ admits a singular Hermitian metric $h$ with Griffiths semi-positive curvature.
Then $\CE$ is weakly positive at $y$ in the sense of Nakayama, for every $y\in Y$ where $\CE$ is locally free around $y$ and $\det h(y)<+\infty$.
\end{thm'}

Going back to the situation in Theorem \ref{MT}, the basic result is due to Berndtsson \cite[1.2]{B}: it states 
that 
the Hermitian metric on $f_{\star}(K_{X/Y}+L)|_{Y_{0}}$ obtained as a family of fiberwise canonical $L^{2}$-metrics has Nakano semi-positive curvature provided the metric $h$ is smooth semi-positive.
We then extend this result (but weakening to the semi-positivity in the sense of Griffiths) to $f_{\star}(K_{X/Y}+L)$ with possibly singular Hermitian line bundle $(L,h)$ with mild singular $h$, for example with trivial multiplier ideal $\CI(h) =\CO_{X}$ (see \ref{ext}).
Finally, regarding $mK_{X/Y}+L=K_{X/Y}+L_{m-1}$ as an adjoint type bundle for $L_{m-1}=(m-1)K_{X/Y}+L$ with a singular Hermitian metric $h_{m-1}$ as we explained above, we deduce Theorem \ref{MT}(2) from all these arguments mentioned above (see \ref{mcor}).

\medskip

The organization of this paper is as follows.
We will start with \S \ref{SsHm} a general discussion on singular Hermitian metrics on vector bundles, its curvature properties and its geometric consequences.
The positivity of relative adjoint type bundles $K_{X/Y}+L$ has been studied in \cite{BPDuke}.
We then proceed to the positivity of $f_{\star}(K_{X/Y}+L)$ in \S \ref{Sadj}, the positivity of $mK_{X/Y}+L$ in \S \ref{Srcb}, and finally the positivity of $f_{\star}(mK_{X/Y}+L)$ in \S \ref{Sdrc}.
Some variants and refinements of Theorem \ref{MT} (1) and (2) will be given in \S \ref{SAfew} and \S \ref{Sdrc2} respectively.

\medskip

Acknowledgments|{\rm Part of the questions addressed in this article were brought to our attention by L. Ein and R. Lazarsfeld; we are very grateful to them for this. 
The first named author would also like to thank F. Ambro, Y. Kawamata, J. Koll\'ar, M. Popa and H. Raufi for interesting and stimulating discussions about various subjects.
The second named author would like to thank Professor Noboru Nakayama for answering questions.
Both authors are deeply indebted to Bo Berndtsson for his constant help and encouragements during the preparation of the present work and for allowing us to include some unpublished material in \cite{BP10}.
This research of S. Takayama is supported by Grant-in-Aid for Scientific Research (B)23340013.}

\newpage

\section{Singular Hermitian metrics on vector bundles}\label{SsHm}

In this first part of our article we introduce a notion of \emph{singular Hermitian metric} on a vector bundle (and more generally, on a torsion free sheaf) and discuss the corresponding notion of positivity/negativity of the associated curvature.
There are earlier attempts to do so, cf. \cite{dC}, \cite{BPDuke}, \cite{Raufi1}, \cite{Raufi2} to quote only a few.

It turns out that as soon as the rank of vector bundle is greater than 1, we have at our disposal two notions of positivity:\ in the sense of Griffiths and in the sense of Nakano, respectively. 
The former corresponds to the type of hypothesis one has in algebraic geometry and has nicer functorial properties, whereas the later is important when one wants to establish vanishing results for higher cohomology. 
The observation which is central in the present article is that ``in practice" what we actually need is to define a notion of negativity (or positivity) of a Hermitian vector bundle $(E, h_E)$ in order to derive some consequences of this property. 

Our definition here is tailored in order to be compatible with the notion of
\emph{weak positivity} in algebraic geometry; more precisely, one of the results we 
want to hold true is that a torsion free sheaf which can be endowed with a singular Hermitian metric
with semi-positive curvature is automatically weakly positive.

\subsection{Prelude}

To start with, let $X$ be a complex manifold, and let $E\to X$ be a vector 
bundle of rank $r\geq 1$, endowed with a (smooth) Hermitian metric $h$. We denote by 
$$\Theta_h(E)\in {\mathcal C}^\infty_{1,1}\big(X, {\rm End}(E)\big)$$
the curvature form of $(E, h)$.
We refer to \cite[Ch.\ 3, Ch.\ 10]{Dnote} for basics.
\smallskip

\noindent We recall next two notions of positivity, playing a important role in analytic and algebraic geometry respectively. The former is the notion of positivity in the sense of Griffiths, and the later is its (conjectural)
counterpart in algebraic geometry.

\begin{dfn} Let $x_0\in X$ be a point, and let $(z_1,..., z_n)$ be a system of local coordinates on $X$ centered at 
$x_0$. We consider a local holomorphic frame $e_1,..., e_r$ of $E$ near $x_0$, such that it is orthonormal at $x_0$.  
The Chern curvature tensor can be locally expressed as follows
$$\Theta_h(E)= \sqrt{-1}\sum c_{j\ol k \lambda\ol \mu}dz_j\wedge d\ol z_k\otimes e_\lambda^\star\otimes e_\mu$$
where $j, k= 1,..., n$ and $\lambda, \mu= 1,..., r$. We say that $(E, h)$ is semi-positive in the sense of Griffiths
at $x_0$ if 
$$
	\sum c_{j\ol k \lambda\ol \mu}\xi_j\ol \xi_kv_\lambda\ol v_\mu\geq 0
$$
for every $\displaystyle \xi\in T_{X, x_0}$ and $v\in E_{x_0}$. 
The Hermitian bundle $(E, h)$ is semi-positive in the sense of Griffiths if it satisfies the property above at each point of $X$.
If there exists a Hermitian form $\omega$ on $X$ such that
$$
	\Theta_h(E)\geq \omega\otimes \text{Id}_{\rm End(E)},
$$
then we say that $(E, h)$ is positive (in the sense of Griffiths). 
The (semi-)negativity is defined in a similar manner.
\end{dfn}

It is well-known (cf.\ \cite[10.1]{Dnote}) that $(E, h)$ is semi-positive if and only if its dual $(E^\star, h^\star)$ is semi-negative in the sense of Griffiths
(we recall that the notion of Nakano positivity fails to satisfy this important property, although we do not recall the definition).

\begin{rem} We denote by ${\bf P}(E^\star)$ the projectivization of the dual of $E$ (we do not follow here the Grothendieck's convention here), and let $\CO_{{\bf P}(E^{\star})}(1)$ be the corresponding tautological line bundle. 
The metric $h$ induces a metric on $\CO_{{\bf P}(E^{\star})}(1)$ and (almost by definition) we see that if the bundle $(E, h)$ is semi-negative, then
$$ 	\log \vert v\vert_h^2 $$
is plurisubharmonic (psh, for short), for any local holomorphic section $v$ of the bundle $\CO_{{\bf P}(E^{\star})}(1)$.
This remark is crucial from our point of view, since it gives the possibility of \emph{defining} the notion of Griffiths positivity for a bundle $(E, h)$ without referring to the curvature tensor!
\end{rem}

\smallskip

\noindent On the algebraic geometry side, we recall that
a vector bundle $E$ on a projective manifold is said to be \emph{ample} if the associated bundle $\CO_{\BP(E)}(1)$ is an ample line bundle on $\BP(E)$. 
Then it is well-known that if $(E, h)$ is positive, then
$E$ is ample (see \cite[6.1.25]{PAG}); the converse, known as Griffiths conjecture, is a long standing and deemed difficult problem. 
\smallskip

\noindent According to \cite[6.1]{DPS}, a vector bundle $E$ on a projective manifold $X$ is \emph{pseudo-effective} if $\CO_{\BP(E)}(1)$ is a pseudo-effective line bundle, and if the union of all curves $C\subset \BP(E)$ with $\CO_{\BP(E)}(1)\cdot C<0$
is contained in a union of subvarieties which does not project onto $X$.
It is shown that in \cite[6.3]{DPS} that this property is equivalent to the  
fact that $E$ is \emph{weakly positive}: there exist positive integers $k_0$ and $m_1(k)$, together with an ample line bundle $A$ such that the map
$$H^0\big(X, S^{m}(S^{k}E\otimes A)\big)\to S^{m}(S^{k}E\otimes A)$$
is generically surjective, for any $k\geq k_0$ and $m\geq m_1(k)$. 
This notion can be formulated in the context of torsion free sheaves, and it is relevant in the following context. 
\vskip2mm
\begin{thm} \cite{Vi1} Let $f: X\to Y$ be a proper surjective morphism of projective manifolds with connected fibers. If $m$ is a positive integer, then the direct image sheaf $f_\star(mK_{X/Y})$ is weakly positive. 
Here we denote by $mK_{X/Y}$ the $m^{\rm th}$ tensor power of the relative canonical bundle.
\end{thm}
\vskip2mm

Our goal in what follows is twofold. We define a metric property of a Hermitian vector bundle $E$ which 
--in the light of Griffiths conjecture-- should correspond to the notion of the weak positivity
(we will at least show that it implies the weak positivity). 
Also, our metric property should be verified by the direct images $f_\star(mK_{X/Y})$. The candidate we have for this is presented in the next subsections; we start with its local version.

\subsection{Singular metrics on trivial bundles}

Here we suppose that $E = X \times \BC^{r}$ is the trivial vector bundle of rank $r$ on a complex manifold $X$.
We denote by
$H_{r}:=\{A=(a_{i\ol j})\}$ the set of $r \times r$, semi-positive definite Hermitian matrixes.
We endow our manifold $X$ with the Lebesgue measure.
\medskip

\noindent We recall the following notion.

\begin{dfn}\label{sHm}
A {\it singular Hermitian metric} $h$ on $E$ is a measurable map from $X$ to $H_{r}$ satisfying $0<\det h<+\infty$ almost everywhere.
\qed\end{dfn}

\noindent By ``convention", a matrix valued function $h=(h_{i\ol j})$ is measurable (resp.\ smooth, continuous, $\ldots$), if all entries $h_{i\ol j}$ are measurable (resp.\ smooth, continuous, $\ldots$).
We will identify two functions if they coincide almost everywhere (however, there is a risk of confusion that a dis-continuous function can be continuous under this convention, for example).
We note that in the paper \cite[p.\,357]{BPDuke}, one does not requires the condition $0<\det h<+\infty$ almost everywhere in the definition of a \emph{singular Hermitian metric}
on a vector bundle. However, as it was highlighted in \cite{Raufi1}, \cite{Raufi2}, 
this additional condition concerning the determinant is important, as it implies that one can define the notion of
\emph{curvature tensor} associated to a singular Hermitian metric.
A standard notion of singular Hermitian metrics of line bundles is $h=e^{-\vph}$ with $\vph \in L^{1}_{loc}(X,\BR)$. 
Our definition in \ref{sHm} requires less in general.

A local section $v$ of $E$ means an element $v \in H^{0}(U,E)$ on some open subset $U\subset X$ (we always assume $v$ to be holomorphic).
If $h$ is a singular Hermitian metric and $v \in H^{0}(U,E)$, then $|v|_{h} : U \to \BR_{\ge 0}$ is a measurable function given by $|v|_{h}^{2} = {}^{t}v h \ol v=\sum h_{i\ol j}v^{i}\ol{v^{j}}$ ($v={}^{t}(v^{1},\ldots,v^{r})$ is a column vector).
Following \cite[p.\,357]{BPDuke}, we introduce the curvature positivity/negativity as follows.

\begin{dfn}\label{curv}
Let $h$ be a singular Hermitian metric on $E$.
\smallskip

\noindent
(1) $h$ is {\it negatively curved} (or $h$ has Griffiths semi-negative curvature), if for any open subset $U\subset X$ and any $v \in H^{0}(U,E)$, $\log |v|_{h}^{2}$ is psh on $U$.
\smallskip

\noindent
(2) $h$ is {\it a.e.\,negatively curved}, if for any open subset $U\subset X$ and any $v \in H^{0}(U,E)$, there exists a psh function $\psi$ on $U$ such that $\psi=\log |v|_{h}^{2}$ almost everywhere on $U$ (we may say $\log |v|_{h}^{2}$ is a.e.\,psh on $U$).
\smallskip

\noindent
(3) $h$ is {\it positively curved} (or $h$ has Griffiths semi-positive curvature), if the dual singular Hermitian metric $h^{\star}:={}^{t}h^{-1}$ on the dual vector bundle $E^{\star}$ is negatively curved.
\qed\end{dfn}

\noindent When $h$ is smooth and $0<\det h<+\infty$ everywhere (i.e. if $h$ is a Hermitian metric
in the usual sense), the requirement \ref{curv}(1) (resp.\ (3)) is nothing but the classical Griffiths semi-negativity (resp.\ Griffiths semi-positivity) of $h$ (see, \cite[\S 2]{Raufi1}). 
\medskip

We have the following remark, which (hopefully) clarifies some aspects of the notion above.

\begin{rem}\label{a.e.}
(1)
If $h$ is negatively curved, then $(h_{i\ol j}(x)) \in H_r$ for any $x \in X$ as we now see.
For a constant vector $v_{1}={}^{t}(1,0,\ldots,0)$, $|v_{1}|_{h}^{2}=h_{1\ol 1}$ is semi-positive and psh.
Hence it is everywhere defined $h_{1\ol 1}\in [0,+\infty)$ and locally bounded.
 For $v={}^{t}(1,1,0,\ldots,0)$, we have $|v|_{h}^{2}=h_{1\ol 1}+h_{2\ol 2}+2\text{Re}\, h_{1\ol 2}$.
For $v'={}^{t}(1,\ai,0,\ldots,0)$, we have $|v'|_{h}^{2}=h_{1\ol 1}+h_{2\ol 2}+2\text{Im}\, h_{1\ol 2}$.
Thus by the same token, $h_{1\ol 2}$ is everywhere defined, and
$|h_{1\ol 2}|^{2}\le h_{1\ol 1}h_{2\ol 2} \le \frac12(h_{1\ol 1}^2+h_{2\ol 2}^2)$.
\smallskip

\noindent (2)
The argument in (1) also shows that if $h$ is a.e.\,negatively curved, there exists  a unique negatively curved singular Hermitian metric $h'$ such that $h=h'$ almost everywhere.
With a suitable convention, we may need not to distinguish negatively curved and a.e.\,negatively curved metrics. 
In the case of line bundles, we do not need to.
It is only a difference $\vph \in L^{1}_{loc}$ with $\ai\rd\rdb \vph \ge 0$ in the sense of distribution and $\vph$ psh.
We could also defined a.e.\,positively curved singular Hermitian metrics.
However in the line bundle case, we suppose every singular Hermitian metric $h$, whose curvature is bounded from below by a continuous form, has an upper-semi-continuous local weight, i.e., $h=e^{-\vph}$ with quasi-psh $\vph$.
\qed
\end{rem}

\begin{rem}
(1)  
Following \cite[Theorem 1.3]{Raufi1}, we will present next here an example 
of negatively curved metric on the trivial rank 2 bundle $E=\BC \times \BC^2$ on $\BC$, for which the curvature tensor cannot be defined in a reasonable way around the origin. 
Indeed, the coefficients of connection matrix corresponding to the next metric 
$$
h=\begin{pmatrix}
1+|z|^{2}& z \\
\ol z& |z|^{2}
\end{pmatrix}
$$
are not in $L^1(\BC, 0)$. 
\smallskip


\noindent
(2) One could define a singular Hermitian metric $h$ on $E$ by requiring that the map $h:X\to H_{r}$ is measurable, such that $\log \det h \in L^{1}_{loc}(X,\BR)$.
The line bundle case, this is the definition. 
However, as the previous example ($h$ with $\det h = |z|^4$) shows it, the curvature tensor 
corresponding to such a metric is not well-defined.
\qed
\end{rem}

\noindent We collect some basic local properties.
The first and most basic one is the following approximation result.

\begin{lem}\label{appr} \cite[3.1]{BPDuke} 
Suppose $X$ is a polydisc in $\BC^{n}$, and suppose $h$  is a singular Hermitian metric on $E$ with negative (resp.\ positive) curvature.
Then, on any smaller polydisc there exists a sequence of smooth Hermitian metrics $\{h_{\nu}\}_{\nu}$
decreasing (resp.\ increasing) pointwise to $h$ whowe corresponding curvature tensor is Griffiths negative (resp.\ positive).
\end{lem}

In the statement above the sequence $\{h_{\nu}\}_{\nu}$ is called \emph{decreasing} in the sense that
the sequence of functions $\{|s|_{h_{\nu}}^{2}\}_{\nu}$ is decreasing for any constant section $s$, or equivalently, the matrix corresponding to each difference $h_{\nu}-h_{\nu+1}$ is semi-positive definite.
We note that $\{h_{\nu}\}_{\nu}$ is obtained simply by convolution (which is possible that to the 
hypothesis concerning $X$).
\smallskip

\noindent This result is relevant for example in the following context.

\begin{lem}\label{det} \cite[Proposition 1.1(ii)]{Raufi1}
Suppose a singular Hermitian metric $h$ is negatively curved.
Then $\log\det h\in L^{1}_{loc}(X,\BR)$ and psh.
\end{lem}

\begin{proof}
Let $\{h_{\nu}\}_{\nu}$ be the sequence of metrics approximating $h$, with the properties stated in
\ref{appr}. We define $\vph_{\nu}:=\log\det h_{\nu}$, the weight of the induced metric on the determinant bundle; it is psh (and smooth), and $\vph_{\nu}$ is decreasing to $\vph_{h}:=\log\det h$.
Thus $\vph_{h}$ is psh and in particular $\vph_{h} \in L^{1}_{loc}(X,\BR)$.
\end{proof}

\noindent As we have already mentioned in \ref{a.e.}, a negatively curved singular Hermitian metric is well-defined in
each point of $X$, thus one can define the set of its eigenvalues (with respect to any smooth
positive definite Hermitian metric). We present in our next statement a few basic properties of these functions.

\begin{lem} \label{eigenbd}
Let $h=(h_{i\ol j})$ be a singular Hermitian metric on $E$.
Let $U\subset X$ be a relatively compact open subset.

\noindent {\rm (1)} Suppose $h$ is negatively curved.
Then there exists a constant $C>0$ such that $|h_{i\ol j}|\le C$ on $U$ for any pair of indices $i,j$.
Let $0\le\lam_{1}(x)\le \lam_{2}(x)\le\ldots \le\lam_{r}(x)< +\infty$ be the eigenvalues of $h(x)$ at each point $x\in X$. Then we have $\lam_{r}(x)\le rC$ on $U$,
$\lam_{1}(x)\ge \frac1{(rC)^{r-1}} \det h(x)$ on $U$, and
$h-\frac1{(rC)^{r-1}} (\det h)I_{r}$ is a measurable map from $U$ to $H_{r}$, where $I_{r}$ is the $r\times r$  identity matrix and $H_r$ is the set of semi-positive Hermitian forms.

\noindent {\rm (2)} Suppose $h$ is positively curved.
Then there exists a constant $C'>0$ (independent of $x\in U\sm \{\det h^{\star}=0\}$) such that $\lam_{1}(x)\ge \frac1{rC'}$ and $\lam_{r}(x)\le (rC')^{r-1} \det h(x)$ on $U \sm \{\det h^{\star}=0\}$, where $0\le\lam_{1}(x)\le \lam_{2}(x)\le\ldots \le\lam_{r}(x)<+\infty$ are eigenvalues of $h(x)$ at $x\in X\sm \{\det h^{\star}=0\}$.
\end{lem}

\begin{proof}
The first assertion of (1) is a direct consequence of our remark \ref{a.e.}.
Given a matrix $A=(a_{ij}) \in M(r,\BC)$, we denote by $C:=\max\{|a_{ij}|;\  1\le i,j\le r\}$, 
the maximum of the absolute value of its entries.
Then we have $|\lam|\le rC$ for any eigenvalue of $A$.
Note $\det h=\lam_{1}\lam_{2}\ldots\lam_{r}\le \lam_{1}(rC)^{r-1}$ at every $x\in U$.

(2) Take dual and apply (1) as long as $\det h^{\star}\ne 0$.
Even if $\det h^\star(x)=0$, we can say $\lam_{1}(x)\ge \frac1{rC'}$ and $\lam_{r}(x)= (rC')^{r-1} \det h(x)=+\infty$ under some convention as in the case of line bundles.
It can be $\lam_{1}(x)=+\infty$ in general for $x\in \{\det h^\star=0\}$.
\end{proof}

\subsection{Global properties}

\noindent We consider in this subsection a holomorphic vector bundle $E$ of rank $r$ on a complex manifold $X$.
 
\begin{dfn}\label{sHm2}
(1) 
A {\it singular Hermitian metric} $h$ on $E$ is a collection of singular Hermitian metrics on local trivializations $E|_{U_{\lam}}\cong U_{\lam}\times \BC^{r}$ satisfying the gluing condition as the one in the usual case.
\smallskip

\noindent (2) 
A singular Hermitian metric $h$ on $E$ is {\it negatively curved} (or $h$ has Griffiths semi-negative curvature), if $\log |v|_{h}^{2}$ 
is psh for any local section $v$ of $E$; $h$ is {\it positively curved} (or $h$ has Griffiths semi-positive curvature), if $h^{\star}={}^{t}h^{-1}$ defines a singular Hermitian metric on $E^{\star}$ with negative curvature.
\smallskip

\noindent (3)
A vector bundle $E$ is negatively (resp.\ positively) curved, if it admits a negatively (resp.\ positively) curved singular Hermitian metric $h$. Unless explicitly stated otherwise, the metric $h$ is allowed to be singular in the sense of (1) of the present definition. \qed \end{dfn}

\medskip

\noindent 
We suppose as an example that $X$ is projective, and $E = A\oplus B$, where $A$ is a positive line bundle and $B$ is a negative line bundle.
As a singular Hermitian metric of $E$ in the weaker sense \cite[p.\,357]{BPDuke}, we can take a direct sum of Hermitian metric $h=h_{A}\oplus h_{B}$ with $h_{A}\equiv 0$ and $h_{B}$ smooth and negative curvature.
Then $h$ is ``negatively curved''.
This is a good indication that indeed one should add the condition $\det h>0$ a.e.\,in the definition of singular Hermitian metric.
\smallskip

\noindent We will start next a systematic study of the properties of negatively curved vector bundles
as above. Our first statement here concerns the 
behavior with respect to inverse images.

\begin{lem}\label{pullback}
Let $f:Y\to X$ be a holomorphic surjective map between two complex manifolds, and let $E$ be a vector bundle on $X$.
\smallskip

\noindent {\rm (1)} We assume that $E$ admits a singular Hermitian metric $h$ with negative (resp.\ positive) curvature. Then $f^{\star}h$ is a singular Hermitian metric on $f^\star (E)$, and it is negatively (resp.\ positive) curved.
\smallskip

\noindent {\rm (2)}
Suppose moreover that $f$ is proper.
Let $X_{1}$ be a non-empty Zariski open subset, and let $Y_{1}=f^{-1}(X_{1})$ be its inverse image.
We consider $h_{1}$ a singular Hermitian metric on $E_{1}=E|_{X_{1}}$.
We assume that the singular Hermitian metric $f^{\star}h_{1}$ on $f^{\star}E_{1}$ extends as a singular Hermitian metric $h_{Y}$ on the inverse image bundle $f^{\star}E$ with negative (resp.\ positive) curvature. Then $h_{1}$ extends as a singular Hermitian metric on $E$ with negative (resp.\ positive) curvature.
In particular, if we have $X_{1}=X$ then the metric $h_{1}$ is negative (resp.\ positive) curved if and only if $f^{\star}h_{1}$ is.
\end{lem}

\noindent We remark that the statement above holds even if we do not assume that the map $f$ is surjective, provided that the pull-back $f^{\star}h$ is well-defined.

\begin{proof}
We will discuss the negatively curved case.
\smallskip

\noindent (1) 
Let $U$ be a coordinate open subset in $X$ such that $E|_{U} \cong U \times \BC^{r}$, and $U' \subset$ a smaller open subset so that an approximation result \ref{appr} holds for $h$.
Let $\{h_{\nu}\}$ be a decreasing sequence towards $h$ as in \ref{appr}.
We consider an open subset $V \subset f^{-1}(U')$ and a non-zero section $v \in H^{0}(V,f^{\star}E)$.
Then the sequence $\{\log |v|_{f^{\star}h_{\nu}}^{2}\}_{\nu}$ is decreasing to $\log |v|_{f^{\star}h}^{2} \not\equiv -\infty$.
Since $f^{\star}h_{\nu}$ is smooth and negatively curved, $\log |v|_{f^{\star}h_{\nu}}^{2}$ is psh.
Thus so is its limit $\log |v|_{f^{\star}h}^{2}$. 
\smallskip

\noindent (2)
Let $u\in H^{0}(U,E)$ be any local section.
Then $f^{\star}u\in H^{0}(f^{-1}(U), f^{\star}E)$.
Then by assumption, $\log |f^{\star}u|_{h_{Y}}^{2}$ is (a.e.)\ psh on $f^{-1}(U)$.
In particular $\log |f^{\star}u|_{h_{Y}}^{2}$ is bounded from above on any relatively compact subset of $f^{-1}(U)$.
For any $x \in U\cap X_{1}$ and a point $y \in f^{-1}(x)$, we have 
$|f^{\star}u|_{h_{Y}}^{2}(y) = |u|_{h_{1}}^{2}(x)$.
Thus we see that the psh function $\log |u|_{h_{1}}^{2}$ a-priori defined only on $U\cap X_{1}$ is bounded from above on $U' \cap X_1$, where $U'\subset U$ is any relatively compact subset. 
Hence, it can be extend as a psh function on $U$, 
so the metric $h_1$ indeed admits an extension $h$ on $X$ in such a way that the Hermitian bundle $(E, h)$ is negatively curved.
\end{proof}

\noindent The next statements show that the tensor bundles associated to $E$ and the quotient bundles of $E$ inherits its 
positivity properties.

\begin{lem}
Let $(E, h_E)$ is positively curved vector bundle. 
Then the symmetric power $S^{m}E$ and the wedge product $\Lambda^{q}E$ are positively curved.
\end{lem}

\begin{proof}
The mechanism of the argument is the same as in \ref{det} (the proof of \cite[Proposition 1.1(ii)]{Raufi1}), which shows $(\det E, \det h)$ is positive.
Basically the statement is well-known for the smooth case, and we are using \ref{appr} in order to 
obtain the general case. We will not provide further details here.
\end{proof}

\begin{lem}\label{subquot}
Let $h$ be a singular Hermitian metric on $E$.
\smallskip

\noindent {\rm (1)} Let $S \subset E$ be a subbundle.
Then the restriction $h_S:=h|_S$ defines a singular Hermitian metric on $S$, and $h_S$ is negatively curved if $h$ is.
\smallskip

\noindent {\rm (2)}
 Let $E\to Q$ be a quotient vector bundle.
Suppose that $h$ is positively curved.
Then $Q$ has a naturally induced singular Hermitian metric $h_Q$ with positive curvature.
\end{lem}

\noindent We will not present here the arguments for the \ref{subquot}, since we will
deal with this kind of statements in a more general context in \ref{subquot2}.
We rather discuss next the connection between \ref{sHm2} and the notion of positivity in algebraic geometry.

\begin{prop}\label{Gp imply wp}
Let $\pi : \BP(E) \to X$ be the projective space bundle, and let $\pi^{\star}E \to \CO_{\BP(E)}(1)$ be the universal quotient line bundle.
Suppose $E$ admits a positively curved singular Hermitian metric $h$.
Then
\smallskip

\noindent {\rm (1)} The bundle
 $\CO_{\BP(E)}(1)$ has a \emph{(naturally induced)} singular Hermitian metric $g$ whose curvature current is semi-positive.
\smallskip

\noindent {\rm (2)} Let $\det h$ be the induced singular Hermitian metric of $\det E$ with semi-positive curvature.
Let $U\subset \BP(E)$ and $W \subset X$ be coordinate neighborhoods respectively such that $U \subset \pi^{-1}(W)$, and regard $g|_{U}$ and $\det h|_{W}$ as a function on $U$ and $W$ respectively.
Then 
$$
	g \, \le \, C\,\pi^{\star}(\det h|_{W})
$$ 
holds on $U$ for a constant $C=C_{U,W}>0$.
\smallskip

\noindent {\rm (2$^\prime$)} 
In particular, if $\det h(x)<+\infty$ at a point $x\in X$, then the restriction $g|_{\BP(E_{x})}$ is well-defined and the multiplier ideal sheaf satisfies $\CI(g^{k}|_{\BP(E_{x})}) =\CO_{\BP(E_{x})}$ for any $k>0$, where $\BP(E_{x})=\pi^{-1}(x)$.
We can express this equality in global terms, namely $\pi(V\CI(g^{k})) \subset V\CI((\det h)^{k})$ holds for any $k>0$, where $V\CI(g^{k})\subset \BP(E)$ is the complex subspace defined by the multiplier ideal sheaf $\CI(g^{k})$.
\smallskip

\noindent {\rm (3)} 
If $X$ is projective, then the inclusion 
$$\mathbf{B}_{-}(\CO_{\BP(E)}(1)) \subset \cup_{k>0}V\CI(g^{k})$$ 
holds, where $\mathbf{B}_{-}(\CO_{\BP(E)}(1))$ is the restricted base locus of $\CO_{\BP(E)}(1)$
introduced in \cite[\S1]{ELMNP06}. 
In particular, $E$ is pseudo-effective in the sense of \cite[6.1]{DPS}, and therefore $E$ is weakly positive in the sense of Viehweg.
\end{prop}

\begin{proof}
\smallskip

\noindent {\rm (1)} 
We will use the notation $\CO(1)$ instead of $\CO_{\BP(E)}(1)$. Our first remark is that 
$\pi^{\star}h$ defines a singular Hermitian metric on $\pi^\star E$ with positive curvature 
(cf. \ref{pullback}). By the property \ref{subquot}, the metric $g^{\star}$ on $\CO(-1)$
induced by the bundle injection $\CO(-1) \to \pi^{\star}E^{\star}$
has negative curvature, hence $\big(\CO(1), g\big)$ is positively curved in the sense of currents,
where $g$ is the dual of the metric $g^\star$.

\smallskip

\noindent {\rm (2)} 
We may suppose that $X$ is a coordinate open set, and that $\det E =X\times \BC$ is trivial. Then we write $\det h=e^{-\vph_{h}}$ where $\vph_{h}$ is a psh function. In this case we also have $\BP(E)=X\times \BP^{r-1}$ where $r=\rank E$.
We consider the set $U:=X\times U_{0}$, where $U_{0}\subset \BP^{r-1}$ is a coordinate open 
set.
Since $h^{\star}$ and $g^{\star}$ are negatively curved, they are bounded in the sense of 
\ref{eigenbd}(1).
Let $x\in X$ be an arbitrary point.
The metric $g^{\star}$ is defined as the restriction of $\pi^{\star}h^{\star}$ to a 1-dimensional linear subspace, therefore we have $\lam_{1}(g^{\star},p)\ge \lam_{1}(\pi^{\star}h^{\star},p)$ for any $p\in \BP(E_{x})\cap U$, where $\lam_{1}(g^{\star},p)$ stands for the smallest eigenvalue of $g^{\star}$ at $p$ (and $\lam_{1}(\pi^{\star}h^{\star},p)$ for $\pi^{\star}h^{\star}$).
We have $\lam_{1}(g^{\star},p)=e^{\psi(p)}$, and $\lam_{1}(\pi^{\star}h^{\star},p)=\lam_{1}(h^{\star},x)\ge \frac1{(rC)^{r-1}}\det h^{\star}(x)$ by \ref{eigenbd} (1), where $C>0$ is a constant independent of $x\in X$.
Thus we have $e^{\psi(p)}\ge \frac1{(rC)^{r-1}}e^{\vph_{h}(x)}$ for any $p\in \BP(E_{x})\cap U$.

\smallskip

\noindent {\rm (2$^\prime$)} 
The first assertion is clear from (2).
For the second, what we proved in (2) above is that $g=e^{-\psi}g_{1}\le (rC)^{r-1} e^{-\pi^{\star}\vph_{h}}$ on $\BP(E)=X\times \BP^{r-1}$ (up to -irrelevant- choices of local trivializations). Since $\pi:\BP(E)\to X$ is a product, our claim is a consequence of Fubini theorem.
\smallskip

\noindent {\rm (3)} 
The pseudo-effectivity of $E$ follows from the properties (1)--(2$^\prime$), combined with the usual techniques in algebraic/analytic geometry.
Then the weak positivity follows from \cite[6.3]{DPS} as a consequence.
\end{proof}

\begin{rem} 
It does not seem to be know wether the weak positivity of $E$ implies the Griffiths semi-positivity 
in the sense of \ref{sHm2}
(this would be a singular version of Griffiths conjecture). 
\qed
\end{rem}
\medskip

\subsection{Metrics on torsion free sheaves}

It turns out that very interesting objects, such as direct images of adjoint bundles $\CE_m:= f_{\star}(mK_{X/Y}+L)$ do not always have a vector bundle structure; nevertheless,
we would like to have a notion of ``singular Hermitian metric" on $\CE_m$.
More generally, we introduce in this subsection a notion of (metric) positivity 
for torsion free sheaves. It turns out that
the theory is essentially the same as the vector bundle case.
\smallskip

\noindent Let $\CE$ be a torsion free (always coherent) sheaf on a complex manifold $X$. 
We will denote by $W_{\CE}\subset X$ the maximum Zariski open subset of $X$ such that
the restriction of $\CE$ to $W_{\CE}$ is locally free.
Then we have $\codim (X\sm W_{\CE}) \ge 2$, thanks to which the following definition is
meaningful.

\begin{dfn}\label{sHm3}
(1)
A {\it singular Hermitian metric} $h$ on $\CE$ is a singular Hermitian metric on the vector bundle $\CE|_{W_{\CE}}$.
\smallskip

\noindent (2)
A singular Hermitian metric $h$ on $\CE$ is {\it negatively curved} (resp.\ {\it a.e.\,negatively curved}), if $h|_{W_{\CE}}$ is negatively curved (resp.\ a.e.\,negatively curved) in the sense of \ref{curv}.
\smallskip

\noindent
(3)
A singular Hermitian metric $h$ on $\CE$ is {\it positively curved}, if there exists a negatively curved singular Hermitian metric $g$ on $\CE^{\star}|_{W_{\CE}}$ such that $h|_{W_{\CE}}=(g|_{W_{\CE}})^{\star}$.
\smallskip

\noindent
(4)
$\CE$ is {\it negatively} (resp.\ {\it positively}) {\it curved}, if it admits a negatively (resp.\ positively) curved singular Hermitian metric.
\qed
\end{dfn}
In what follows, we may equally say that \emph{$h$ has Griffiths semi-positive/negative curvature}
meaning that the sheaf in question has the property (2) or (3) in the definition above.

\begin{rem}\label{dual}
(1)
If $W'\subset W_{\CE}$ is a non-empty Zariski open and $h'$ is a singular Hermitian metric on $\CE|_{W'}$, then $h'$ can be seen as a singular Hermitian metric on $\CE|_{W_{\CE}}$ by assigning the (arbitrary) value $h'\equiv 0$ on $W_{\CE}\sm W'$.
In particular, there is a natural correspondence between the set of singular Hermitian metric on $\CE$ and the set of singular Hermitian metric on $\CE^{\star\star}$, via $\CE|_{W_{\CE}} \cong \CE^{\star\star}|_{W_{\CE}}$.
\smallskip

\noindent (2)
If $h$ is a negatively curved singular Hermitian metric on $\CE$, then for any open $U\subset X$ and any $v \in H^{0}(U,\CE)$, $\log |v|_{h}^{2}$ is psh on $U\cap W_{\CE}$.
Because of $\codim (X\sm W_{\CE}) \ge 2$, $\log|v|_{h}^{2}$ extends as a psh function on $U$ (compare with \ref{curv}). Thus in the case under discussion, 
we have \emph{unique privileged} extension of the metric $h$, once it is defined on the complement of a set
of codimension at least two.
\smallskip

\noindent (3)
In order to obtain a negatively curved singular Hermitian metric on $\CE$, it is enough to find an a.e.\,negatively curved singular Hermitian metric as in \ref{a.e.}(2). 
In conclusion, we see that $\CE$ is negatively (resp.\ positively) curved 
if and only if its associated double dual $\CE^{\star\star}$ is negatively (resp.\ positively) curved. 
\qed
\end{rem}

\begin{lem}\label{subquot2}
Let $h$ be a singular Hermitian metric on $\CE$.
\smallskip

\noindent {\rm(1)} Let $\CS \subset \CE$ be a subsheaf.
Then the restriction $h_\CS:=h|_\CS$ defines a singular Hermitian metric on $\CS$. 
If $h$ is a.e.\,negatively curved, then so is $h_\CS$.
In particular $\CS$ is negatively curved if $\CE$ is, by {\rm \ref{dual} (3)}.
\smallskip

\noindent {\rm (2)} Let $\CE\to \CQ$ be a quotient torsion free sheaf.
Suppose that $h$ is positively curved.
Then $\CQ$ has a naturally induced singular Hermitian metric $h_\CQ$ with positive curvature.
\smallskip

\noindent {\rm (3)} Let $\CF$ be a torsion free sheaf, and suppose there exists a sheaf homomorphism $a : \CE\to\CF$ which is generically surjective.
Suppose that $h$ is positively curved.
Then $\CF$ has a naturally induced singular Hermitian metric $h_{\CF}$ with positive curvature.
\end{lem}

\begin{proof}
(1) 
By restricting everything on the maximum Zariski open subset where $\CS$ is locally free, we may assume $\CS$ is locally free.
Let $W$ be a Zariski open subset such that $\CE$ is locally free and $\CS$ is a subbundle of $\CE$.
Then $h_\CS=h|_\CS$ defines a singular Hermitian metric on $\CS$ over $W$, and hence induces a singular Hermitian metric on $\CS$ over $X$ by \ref{dual}.
Let $U$ be an open subset and $v \in H^{0}(U,\CS) \subset H^{0}(U,\CE)$.
Then on $U\cap W$, $|v|_{h_{\CS}}^{2}=|v|_{h}^{2}$ (in the right hand side, $|v|_{h}^{2}$ is measured as local section of the sheaf $\CE$, i.e.\ $v\in H^{0}(U,\CE)$).
If $h$ is a.e.\,negatively curved, then $\log |v|_{h}^{2}$ is a.e.\,psh on $U$, and hence $h_{\CS}$ is a.e.\,negatively curved.
\smallskip 

\noindent (3)
By dualyzing the map $a$, we obtain a sheaf injection $\CF^{\star}\to \CE^{\star}$.
Thanks to the point (1), we obtain a singular Hermitian metric $h^{\star}|_{\CF^{\star}}$ on $\CF^{\star}$ with a.e.\,negative curvature.
The procedure described in \ref{dual} allows us to modify $h^{\star}|_{\CF^{\star}}$ on a Zariski closed set, such that  the result is
a singular Hermitian metric with negative curvature on $\CF^{\star}$.
By taking the dual, we have a singular Hermitian metric $(h^{\star}|_{\CF^{\star}})^{\star}$ (not exactly same, but the discrepancies eventually occur on a measure zero set only) on $\CF^{\star\star}$ with positive curvature.
We then use \ref{dual} again to obtain a singular Hermitian metric on $\CF$.
\end{proof}

\subsection{Griffiths semi-positivity and weak positivity}

\noindent 
We show here that a torsion free sheaf endowed with a singular Finsler metric with positive curvature with a condition on the singular locus is \emph{weakly positive in the sense of Nakayama}. To this end we will
first recall this notion due to Nakayama which is a refinement of the corresponding definition due to Viehweg.

For a torsion free sheaf $\CF$ on a variety $Y$, we denote by $S^{m}(\CF)$ the $m$-th symmetric tensor product of $\CF$ with the convention that $S^{0}(\CF) = \CO_{Y}$, and let $\what{S}^{m}(\CF)$ be the double dual of the sheaf $S^{m}(\CF)$.

\begin{dfn}\label{dd-ample} \cite[V.3.20]{Nbook}
Let $Y$ be a smooth projective variety, and let $\CF$ be a torsion free coherent sheaf on $Y$.
Let $A$ be an auxiliary ample divisor on $Y$.
Then
\smallskip 

\noindent 
(1) $\CF$ is {\it weakly positive at a point $y \in Y$}, if for any integer $a > 0$, there exists an integer $b>0$ such that $\what S^{ab}(\CF)\ot\CO_{Y}(bA)$ is generated by global sections at $y$. 
\smallskip 

\noindent 
(2) $\CF$ is {\it dd-ample at a point $y \in Y$} (``ample modulo double-duals''), if $\what{S}^{m}(\CF)\ot\CO_{Y}(-A)$ is generated by global sections at $y$ for an integer $m>0$.

\smallskip 

\noindent 
(3) $\CF$ is {\it weakly positive} (resp.\ {\it big}), if it is weakly positive (resp.\ dd-ample) at some point.
\qed
\end{dfn}

We note that $\CF$ is weakly positive over a Zariski open subset $U \subset Y$ in the sense of Viehweg \cite[2.13]{Vbook} implies that $\CF$ is weakly positive at every point $y \in U$ in the sense of Nakayama (\cite[p.\,200, Remark (4)]{Nbook}). Thus the weak positivity in the sense of Viehweg 
requires global generation property (1) to hold on a Zariski open subset, whereas in the version due to Nakayama the property (1) above could be verified on countable intersection of Zariski open sets only.

If $\CF$ is a line bundle, we set $\Amp(\CF)=\{y\in Y;\ \CF$ is dd-ample at $y \}$.
The complement $\NAmp(\CF)=Y\sm \Amp(\CF)$ is commonly denoted by
${\bf B}_+(\CF)$ and called the {\it augmented base locus} of $\CF$ (\cite[\S 1]{ELMNP09}).
\medskip

\noindent The following is a generalization of \ref{Gp imply wp}, and its proof goes along the same lines.
So we leave it to the interested reader (see \ref{bigness} for a more general form of this result).

\begin{thm}\label{Gp imply wp2}
Let $Y$ be a smooth projective variety, and let $\CE$ be a torsion free sheaf on $Y$.
Suppose that $\CE$ admits a singular Hermitian metric $h$ with positive curvature.
Then $\CE$ is weakly positive at $y$ (in the sense of Nakayama \ref{dd-ample}), for every $y\in Y$ where $\CE$ is locally free around $y$ and $\det h(y)<+\infty$.
\end{thm}
\medskip

\noindent The rest of this subsection is devoted to the proof the following result.

\begin{thm} \label{bigness}
Let $Y$ be a smooth projective variety, and let $\CF$ be a torsion free coherent sheaf on $Y$.
Let $\BP(\CF) = \Proj (\bigoplus_{m \ge 0} S^m(\CF))$ be a scheme over $Y$ associated to $\CF$, say $\pi : \BP(\CF) \to Y$, and let $\CO_{\CF}(1)$ be the tautological line bundle on $\BP(\CF)$.
Let $Y_1 \subset Y$ be a Zariski open subset on which $\CF$ is locally free (in particular $\BP(\CF)$ is smooth over $Y_1$) and $\codim_{Y}(Y\sm Y_{1}) \ge 2$.
Suppose that $\CO_{\CF}(1)|_{\pi^{-1}(Y_{1})}$ admits a singular Hermitian metric $g$ with semi-positive curvature, and 
that there exists a point $y \in Y_1$ such that $\CI(g^k|_{\BP(\CF_y)})=\CO_{\BP(\CF_y)}$ for any $k>0$, where $\BP(\CF_{y})=\pi^{-1}(y)$.
Then
\smallskip 

\noindent 
{\rm (1)} $\CF$ is weakly positive at $y$.
\smallskip 

\noindent 
{\rm (2)} Assume moreover that there exists an open neighborhood $W$ of $y$ and a K\"ahler form $\eta$ on $W$ such that $\ai\Th_{g} - \pi^{\star}\eta \ge 0$ on $\pi^{-1}(W)$, then $\CF$ is dd-ample at $y$.
In particular $\CF$ is big.
\end{thm}

\medskip

 \noindent As a preliminary discussion, we recall some facts due to \cite[V.3.23]{Nbook}.
 Let $\pi : \BP(\CF) \to Y$ be the scheme associated to $\CF$, together with its tautological 
 line bundle $\CO_{\CF}(1)$.
Let $Y_2 \subset Y$ be the maximum Zariski open subset of $Y$ such that the restriction 
$\CF|_{Y_2}$ is locally free; we have $W \subset  Y_{1} \subset  Y_{2}$.
Let $\BP^\prime(\CF) \to \BP(\CF)$ be the normalization of the component of $\BP(\CF)$ containing $\pi^{-1}(Y_2)$, and let $X \to \BP'(\CF)$ be a birational morphism from a smooth projective variety that is an isomorphism over $\pi^{-1}(Y_2)$.
Let $\mu : X \to \BP'(\CF) \to \BP(\CF)$ be the composition, and let $f=\pi\circ\mu : X \to Y$.
We take a morphism $X \to \BP'(\CF)$ so that $X \sm f^{-1}(Y_{2})$ is a divisor.
We set
$$
	Z=\pi^{-1}(Y_1), \ \ X_{1}=\mu^{-1}(Z)=f^{-1}(Y_{1})
$$
and may regard $Z = X_{1} \subset X$ by $\mu$.
$$
\begin{CD}
    X\  @>{\mu}>>  \BP(\CF) @. \ \supset Z \\
    @V{f}VV    @VV{\pi}V  \ \ \ \ @VVV \\
    Y  @=   Y @. \ \supset \ Y_{1}
\end{CD}
$$
Denote by $L_{0}=\mu^{\star}\CO_{\CF}(1)$ a line bundle (or a divisor) on $X$.
There exists a natural inclusion $\CF \to f_{\star}L_{0}$, which is isomorphic over $Y_{2}$.
By \cite[III.5.10(3)]{Nbook}, there exists an $f$-exceptional effective divisor $E$ such that $f_{\star}\CO_{X}(m(L_{0}+E)) \cong \what S^{m}(\CF)$ for any integer $m>0$.
We fix such a divisor $E$ and let 
$$
	L = L_{0}+E.
$$
Then \ref{bigness} is essentially reduced to the following

\begin{thm} \label{Lbig}
In the situation above and under the assumptions in {\rm \ref{bigness} (2)}, the line bundle $L$ is big and $f^{-1}(y) \subset \Amp(L)$.
\end{thm}

\begin{proof} The rough outline of our argument goes as follows. By the main result of
\cite{ELMNP09}, it is enough to prove that we have 
$$
	\liminf_{\varepsilon\to 0}\vol_{X|V}(L+\varepsilon A)> 0
$$
for any positive-dimensional analytic set $V$ containing a point $x\in f^{-1}(y)$. If the metric $g$ 
would be known to be globally defined and semi-positively curved on the whole space $X$,
this can be shown
to hold true by a simple application of the holomorphic Morse inequalities \cite[\S 8]{Dnote}. However, in our current situation $g$ is only defined on $X_1$, so the existence of sections allowing us to show the inequality above will be shown in a direct manner, by the $L^2$ theory.

\noindent We start by a few general reduction steps.

\noindent (1)
Let us take a point $x \in X$.
Assume that there exist an integer $\ell_{x}>0$ and a (not necessarily ample) line bundle $A$ on $X$ such that the evaluation map to the jet space at $x$;
$$
	H^{0}(X, \CO_X(k\ell_{x}L+A)) 
	\lra \CO_X(k\ell_{x}L+A) \ot \CO_{X}/\fm_{X,x}^{k+1} 
$$
is surjective for any integer $k > 0$.
Here in general $\fm_{X,x}$ stands for the maximal ideal of a point $x$ on a variety $X$.
We then can show that $x \in \Amp(L)$ as follows.

Due to the jet separation property, for any integer $k>0$ and any positive dimensional subvariety $V \subset X$ containing $x$, we have 
$\vol_{X|V}(k\ell_{x}L+A) \ge k^{d}$, where $\vol_{X|V}(k\ell_{x}L+A)$ is the restricted volume and $d = \dim V > 0$ (by \cite[4.1]{ELMNP09} for example).
Thus 
$$\vol_{X|V}(L+\frac1{k\ell_{x}} A) \ge 1/\ell_{x}^{d}$$ holds (we remark here that $l_x$ is independent of $V$).

Assume on the contrary that $x \in \NAmp(L)$.
Since $\NAmp(L)$ has no isolated points (\cite[1.1]{ELMNP09}), there exists a subvariety $V \subset X$ of $d := \dim V > 0$, which contains $x$ and is an irreducible component of $\NAmp(L)$.
By the preceding argument,  we infer that we have
$\liminf_{k\to\infty} \vol_{X|V}(L+\frac1{k\ell_{x}} A) \ge 1/\ell_{x}^{d}>0$.
This contradicts to \cite[5.7]{ELMNP09}.
\smallskip

(2)
We denote by $\BP(\CF_y) = \pi^{-1}(y) \subset Z, X_{y}=f^{-1}(y) \subset X$, and the ideal sheaves $\CI_{\BP(\CF_y)}$ and $\CI_{X_{y}}$ respectively.
We will show in (3) below that there exist an integer $\ell_{y}>0$ and a line bundle $A$ on $Y$ such that the restriction map
$$
(*) \ \ 	J_{y}^{k} : H^{0}(Z, \CO_{\CF}(k\ell_{y})\ot\pi^{\star}A) \lra 
	H^0(\BP(\CF_y), \CO_{\CF}(k\ell_{y}) \ot \pi^*A \ot \CO_{Z}/\CI_{\BP(\CF_y)}^{k+1})
$$
is surjective for any integer $k>0$.
Note that the fibration $\pi : Z \to Y_{1}$ is geometrically simple enough so that we have an isomorphism
$$
	H^0(\BP(\CF_y), \CO_{\CF}(k\ell_{y}) \ot \pi^*A \ot \CO_{Z}/\CI_{\BP(\CF_y)}^{k+1})
	\cong H^0(\BP(\CF_y), \CO_{\BP(\CF_{y})}(k\ell_{y})) 
			\ot_{\CO_{Y,y}} \CO_{Y}/\fm_{Y,y}^{k+1},	
$$
and a similar isomorphism for $f : X_{1}\to Y_{1}$.
Taking the surjection $(*)$ for granted, we show that $f^{-1}(y) =X_{y}\subset \Amp(L)$ as follows.

We first note that every section of $\CO_X(k\ell_{y}L) \ot f^{\star}A$ over $X_{1}(=Z)$ extends to a section of $\CO_X(k\ell_{y}L) \ot f^{\star}A$ over $X$, because
$H^{0}(Z, \CO_{\CF}(k\ell_{y})\ot\pi^{\star}A)
\cong H^{0}(Y_{1}, S^{k\ell_{y}}(\CF) \ot A)
\hookrightarrow H^{0}(Y, \what S^{k\ell_{y}}(\CF) \ot A)
\cong H^{0}(Y,f_{\star}(\CO_X(k\ell_{y}L)) \ot A)
\cong H^{0}(X, \CO_X(k\ell_{y}L) \ot f^{\star}A)$.
Here we need to explain the inclusion and the next isomorphism.
Since $\what S^{k\ell_{y}}(\CF)$ is reflexive on $Y$ and locally free on $Y_{1}$ of $\codim_{Y}(Y \sm Y_{1}) \ge 2$, we have a natural inclusion $H^{0}(Y_{1}, S^{k\ell_{y}}(\CF) \ot A) \hookrightarrow H^{0}(Y, \what S^{k\ell_{y}}(\CF) \ot A)$.
The isomorphisms $f_{\star}(\CO_X(mL)) \cong \what S^{m}(\CF)$ for any $m>0$ (\cite[III.5.10(3)]{Nbook} mentioned above) induce
$H^{0}(Y, \what S^{k\ell_{y}}(\CF) \ot A) \cong H^{0}(Y,f_{\star}(\CO_X(k\ell_{y}L)) \ot A)$.

Thus, from $(*)$, the restriction map
$$
H^{0}(X, \CO_X(k\ell_{y}L) \ot f^{\star}A) \lra 
H^0(X_y, \CO_X(k\ell_{y}L) \ot f^*A \ot \CO_{X}/\CI_{X_y}^{k+1})
$$
is surjective for any $k>0$.
We take an arbitrary point $x \in X_y$ and show that $x \in \Amp(L)$.
Noting $L|_{X_y} \cong \CO_{\BP(\CF_{y})}(1)$, the restriction map
$H^0(X_y, \CO_X(k\ell_{y}L)) \to \CO_{X_y}/\fm_{X_y,x}^{k+1}$
is surjective for any $k>0$.
Thus, as a combination of the above two types of surjections, the restriction map
$$
H^{0}(X, \CO_X(k\ell_{y}L) \ot f^{\star}A) \lra 
 \CO_X(k\ell_{y}L) \ot f^{\star}A \ot \CO_{X}/\fm_{X,x}^{k+1}
$$ 
is surjective for any $k>0$.
Then by (1), we have $x \in \Amp(L)$.
 
(3)
Let us prove the surjectivity of ($*$) in (2).
Denote by $m= \dim Y$.
We take a local coordinate $t = (t_{1}, \ldots, t_{m})$ centered at $y$, and let $|t| = (\sum_{i=1}^{m}|t_{i}|^{2})^{1/2}$.
We take a $C^{\infty}$ cut-off function $\rho : Y \to \BR$ such that $\rho\equiv 1$ on a neighborhood of $y$ and $\supp \rho \subset W$.
We let a function $\phi := (m+1)\rho\log|t|^{2} : Y \to [-\infty,\infty)$.
Comparing with the K\"ahler form $\eta$, we can find an integer $\ell_{y}>0$ such that $\ell_{y}\eta+\ai\rd\rdb\phi>0$ on $W$.
We then compare with the curvature current $\ai\Th_{g}$ of $\CO_{\CF}(1)|_{Z}$.
Since $\ai\Th_{g}\ge\pi^{\star}\eta$ on $\pi^{-1}(W)$, we can also see that 
$\ell_{y}\ai\Th_{g}+\ai\rd\rdb\pi^{\star}\phi$ is a semi-positive current on $Z$.
We take an ample line bundle $A$ on $Y$ such that $A \ot K_Y^{-1} \ot (\what \det \CF)^{-1}$ is ample, where $\what \det \CF$ is the double dual of $\bigwedge^r \CF$ and $r$ is the rank of $\CF$.
Let $h_{K_Y}$ (resp.\ $h_{\what \det \CF}$) be a smooth Hermitian metric on $K_Y$ (resp.\ $\what \det \CF$), and let $h_A$ be a smooth Hermitian metric on $A$ such that $h_A h_{K_Y}^{-1} h_{\what \det \CF}^{-1}$ has positive curvature.

We now take any integer $k>0$.
We consider the following line bundle on $Z$:
$$
	M (=M_{k}) := \CO_{\CF}(k\ell_{y}+r)|_Z 
		\ot \pi^* \big(A \ot K_Y^{-1} \ot (\what \det \CF)^{-1} \big)|_Z.
$$
We note $\CO_{\CF}(k\ell_{y}) \ot \pi^*A = K_Z \ot M$ on $Z$.
The line bundle $M$ has a singular Hermitian metric
$$
  g_M := (g^{\ell_{y}} \pi^*e^{-\phi})^{k} \cdot g^{r}
				\cdot \pi^{\star}(h_A h_{K_Y}^{-1} h_{\what \det \CF}^{-1})
$$
of semi-positive curvature.
We note that 
$$
	\ai\Th_{g_M} \ge r \ai\Th_{g} \ge \pi^*\eta \text{ on } \pi^{-1}(W),
$$	
which is a strict positivity for the normal direction of $\BP(\CF_{y})$ in $Z$.
We also note that we can take an open ball neighborhood $U \subset W$ of $y$ such that $\CI(g^{\ell_y k+r}) =\CO_{\BP(\CF)}$ on $\pi^{-1}(U)$, in particular 
$$
	\CI(g_M)=\CO_{Z} \text{ on } \pi^{-1}(U\sm \{y\})
$$
(as $\ell_y$ and $r$ are given, $U$ depends only on $k$).
This follows from our assumption $\CI(g^{\ell_y k+r}|_{\BP(\CF_y)})=\CO_{\BP(\CF_y)}$ and the restriction theorem of multiplier ideals, namely Ohsawa-Takegoshi's extension.
We also have
$$
	\CI(g_M)
	\subset \CI(e^{-k(m+1)\pi^*\log|t|^2})
	= \CI_{\BP(\CF_y)}^{k(m+1)-m+1}
	\subset \CI_{\BP(\CF_y)}^{k+1} \text{ on } \pi^{-1}(U).
$$

We take a section 
$\sg \in H^0(\BP(\CF_y), \CO_{\CF}(k\ell_{y}) \ot \pi^*A
\ot \CO_{Z}/\CI_{\BP(\CF_y)}^{k+1})$.
Noting
$$
	H^0(\BP(\CF_y), \CO_{\CF}(k\ell_{y}) \ot \pi^*A \ot \CO_{Z}/\CI_{\BP(\CF_y)}^{k+1})
	\cong H^0(\BP(\CF_y), \CO_{\BP(\CF_{y})}(k\ell_{y})) 
				\ot_{\CO_{Y,y}} \CO_{Y}/\fm_{Y,y}^{k+1},	
$$
we have a section
$\sg' \in H^0(\BP(\CF|_{U}), \CO_{\CF}(k\ell_{y}) \ot \pi^*A)$, which is mapped to $\sg$ by the restriction map
$$
	J_{y}^{k} : H^{0}(\BP(\CF|_{U}), \CO_{\CF}(k\ell_{y})\ot\pi^{\star}A) \lra 
	H^0(\BP(\CF_y), \CO_{\CF}(k\ell_{y}) \ot \pi^*A \ot \CO_{Z}/\CI_{\BP(\CF_y)}^{k+1}).
$$
The map $J_{y}^{k}$ is defined for sections over $\BP(\CF|_{Y'})$ for any open $Y' \subset Y$ containing $y$. 
We use the same notation $J_{y}^{k}$ for any open $Y' \subset Y$.
We take a $C^\infty$ cut-off function $\lam$ such that $\lam\equiv 1$ around $y$, say on $U_1 \subset U$, and $\supp \lam \subset U$.
We consider 
$$
	u := \rdb ((\pi^*\lam) \sg') = \rdb (\pi^*\lam) \wed \sg',
$$
which can be seen as an $M$-valued $(n,1)$-form on $Z$ via
$(\CO_{\CF}(k\ell_{y}) \ot \pi^*A)|_Z = K_Z \ot M$, where $n := \dim Z = m+r-1$.
Thanks to the observation on the multiplier ideal $\CI(g_M)$ on $\pi^{-1}(U)$, we see that $u$ is square integrable with respect to $g_M$ and any given complete K\"ahler metric on $Z$.
(Recall that $Z \cong X_{1}$ is a Zariski open subset of a smooth projective variety $X$, in particular $Z$ admits a complete K\"ahler metric, refer \cite[0.2]{D82} if one needs.
The quasi-projectivity of $Z$ will be more important for a regularization process.)
Then by \cite[5.1]{D82}, we can solve a $\rdb$-equation $\rdb v = u$ for some $v \in L_{2,loc}^{n,0}(Z, M)$ (an $M$-valued $(n,0)$-form on $Z$ with locally square integrable functions coefficient) and $c_{n} v\wed \ol v\, g_M \in L^1_{loc}(Z,\BR)$ with $c_{n}=(-1)^{n^2/2}$. 
Since $u \equiv 0$ on $\pi^{-1}(U_{1})$, $v$ is holomorphic on $\pi^{-1}(U_{1})$.
The integrability $c_{n} v\wed \ol v\, g_M \in L^1_{loc}(Z,\BR)$,
in particular $c_{n} v\wed \ol v\,  e^{-k\pi^*\phi} \in L^1_{loc}(Z,\BR)$, implies $v|_{\pi^{-1}(U_{1})} \in 
H^{0}(\BP(\CF|_{U_{1}}), \CO_{\CF}(k\ell_{y})\ot\pi^{\star}A \ot \CI_{\BP(\CF_{y})}^{k+1})$.
Then $\wtil \sg := (\pi^*\lam)\sg' - v \in H^0(Z, K_Z \ot M)
= H^{0}(Z, \CO_{\CF}(k\ell_{y}) \ot \pi^*A)$
and $J_{y}^{k}(\wtil \sg) = J_{y}^{k}(\sg') = \sg$.
This proves the surjectivity of ($*$) in (2). 
\end{proof}

\begin{proof}[Proof of {\rm \ref{bigness}}]
We first prove (2).
By \ref{Lbig} and \cite[V.3.23]{Nbook}, our $\CF$ in (2) is dd-ample at $y$.
Then by definition, $\CF$ is big.

(1) Let $g$ be a singular Hermitian metric on $\CO_{\CF}(1)|_{\pi^{-1}(Y_{1})}$ in the statement.
Let $H$ be an ample line bundle on $Y$ with a Hermitian metric $h_{H}$ with positive curvature.
Let $a>0$ be an arbitrary integer.
By applying the preceding argument in the proof of \ref{Lbig} for a line bundle $\CO_{\CF}(a) \ot \pi^{\star}H$ over $\pi^{-1}(Y_{1})$ with a metric $g^{a}\cdot \pi^{\star}h_{H}$ instead of $\CO_\CF(1)|_{\pi^{-1}(Y_{1})}$ with $g$, we see that $\what S^{a}(\CF) \ot H$ is dd-ample at $y$. 
Note that we could take a K\"ahler form $\eta = \ai\Th_{h_{H}}$ on $Y$.
Since $a>0$ is arbitrary, $\CF$ is weakly positive at $y$ (\cite[p.\,200, Remark (3)]{Nbook}).
\end{proof}
\smallskip

\newpage
\section{Positivity of direct image sheaves of adjoint type}\label{Sadj}

\subsection{The $L^{2/m}$ metric, absolute version}\label{abs}

We recall in this subsection the construction of canonical metrics on the bundle $mK_X+ L$ and the vector space $H^{0}(X,mK_{X}+L)$, due to Narashimhan-Simha's  (\cite[\S 2]{NS}). We will deal with the relative version of this metric in the next subsection.

\begin{no}\label{NSc} 
Let $X$ be a compact complex manifold of $\dim X=n$, and $L$ be a holomorphic line bundle on $X$ with a singular Hermitian metric $h$ with semi-positive curvature.

(1) For every $u\in H^0(X,mK_X+L)$, $(cu\wed \ol u h)^{1/m}$ is a real semi-positive $2n$-form on $X$ with measurable function coefficients, where $c=c_{n,m}=((-1)^{n^{2}/2}(-1)^n)^{-m}$.
It is often written as $|u|^{2/m}h^{1/m}=(cu\wed \ol u h)^{1/m}$. 
We set
$$
	\| u \|_{m} = \left(\int_X (cu\wed \ol u h)^{1/m}\right)^{m/2}.
$$
This is called an $L^{2/m}$-pseudo-norm on $H^0(X,mK_X+L)$ (note that $\|u\|_{p}$ is not an $L^{p}$-norm).
It is a norm (and in fact a metric) only when $m=1$. 
In this case we will refer to $\| \cdot \|_{1}$ as the {\it canonical $L^{2}$-metric}, and set $g_{1}(u,v)=\int_{X} c_{n}u \wed \ol v h$ the corresponding pairing, where $c_{n}=c_{n,1}=(-1)^{n^{2}/2}$.
To be more precise, this is defined on a subspace
$$
	V_m=\{u\in H^0(X,mK_{X}+L);\  \|u\|_{m} <\infty\}.
$$
By definition of $\| u\|_{m}$, if $h^{1/m}\in L^{1}_{loc}$, then $V_m= H^0(X,mK_{X}+L)$.
We will remark in \S \ref{Sm>1} that there exists a coherent ideal sheaf $\CJ_{m}(h^{1/m})$, called an $L^{2/m}$-multiplier ideal sheaf, such that $V_m=H^0(X,(mK_{X}+L) \ot \CJ_{m}(h^{1/m}))$.
Even if our notation does not reflects it, this space depends on $h$.
\smallskip

(2) 
We consider a $C^\infty$ volume form $dV$ on $X$ (i.e.\ a $C^\infty$ Hermitian metric on $-K_X$) and a $C^\infty$ Hermitian metric $h_{0}$ on $L$
as reference metrics.
Then the quantity $cu\wed\ol u h_0/(dV)^m$ represents a $C^\infty$ semi-positive function on $X$.
For every $x \in X$, we set 
$$
 B_m'(x)=\sup\left\{ \frac{cu\wed\ol u h_0}{(dV)^m}(x);\ 
 		u\in H^0(X,mK_X+L), \|u\|_{m} \le 1\right\}.
$$
We obtain a function $B_{m}'$ on $X$, and set $B_m=B_m'(dV)^m h_0^{-1}$.
Here $B_m'$ depends on $dV$ and $h_0$, but $B_m$ does not.
Because of this reason, we often denote by
$$
 B_m(x)= \sup\{ cu\wed\ol u (x);\ u\in H^0(X,mK_X+L), \|u\|_{m} \le 1\}.
$$
This $B_{m}$ defines a singular Hermitian metric on the dual $(mK_X+L)^{\star}$, and is called the {\it (twisted) canonical $L^{2/m}$ metric, or $m$-th Bergman kernel metric}. 
For every $x\in X$ and a vector $\xi_x \in (mK_{X}+L)^\star_x$,
the length with respect to the metric $B_{m}$ is given by
$$
	|\xi_x|_{B_{m}}:=\sup\{|\xi_x \cdot u(x)|;\ u\in H^0(X,mK_{X}+L), \|u\|_{m} \le 1\},
$$
where $\xi_x \cdot u(x)$ is the duality pairing between $(mK_{X}+L)^\star_x$ and $(mK_{X}+L)_x$. 
The dual metric 
$$
	h^{(m)}:=B_{m}^{-1}
$$ 
is a singular Hermitian metric on $mK_{X}+L$, and it is clear that the curvature current of $h^{(m)}$ is semi-positive.

(3)
We define the line bundle
$$
L_{m-1}= (m-1)K_{X}+L=\frac{m-1}m (mK_{X}+L) +\frac1m L
$$
and a singular Hermitian metric
$$
h_{m-1} =(B_{m}^{-1})^{(m-1)/m}h^{1/m}.
$$
The construction above induces a Hermitian form 
$g_{mNS}$ ($g_m$ for short slightly) on the vector space $H^0(X,mK_{X}+L)$ as follows 
$$
	g_{mNS}(u, v) = \int_X c_{n} u\wed \ol v \, h_{m-1}, 
$$
which is nothing but the canonical $L^{2}$-metric on $H^{0}(X,K_{X}+L_{m-1})$ with respect to $h_{m-1}$ as in (1).
This Hermitian form is defined on $H^0(X,(K_{X}+ L_{m-1})\ot \CI(h_{m-1}))$.
In the case $m=1$, $g_{1NS}$ is nothing but the canonical $L^{2}$-metric $g_{1}$ with respect to $h_{0}=h$, i.e.\,$\|\cdot\|_{1}$ in (1).
We may call the Hermitian form $g_{mNS}$ on $H^0(X,mK_{X}+L)$ as the {\it (twisted) $m$-th Narashimhan-Simha Hermitian form}.
In a relative situation $f:X\to Y$, a family of fiberwise $h^{(m)}$ (resp.\ of $g_{mNS}$) will define a metric on $mK_{X/Y}+L$ (resp.\ on $f_{*}(mK_{X/Y}+L)$).

(4)
We note that, although $h_{m-1}$ is ``more singular'' than $h^{1/m}$, the inclusion 
$$
	H^0(X,(mK_{X}+L)\ot \CI(h)) \subset  V_m 
	\subset H^0(X,(K_{X}+L_{m-1})\ot \CI(h_{m-1}))
$$
holds true for any integer $m>0$.
For the left inclusion, $|u|^2h \in L^1_{loc}$ for $u\in  H^0(X,mK_{X}+L)$ implies
$|u|^{2/m}h^{1/m} \in L^1_{loc}$ by H\"older's inequality.
For the right inclusion, it is enough to see that $c_{n}u\wed\ol u\, h_{m-1} \in L^1_{loc}$ after a local trivialization, where $u\in V_{m}\subset H^0(X,mK_{X}+L)$.
By definition of $B_{m}^{-1}$,
$$
	c_{n} u\wed\ol u\, h_{m-1}=c_{n} u\wed \ol u\, (B_{m}^{-1})^{(m-1)/m}h^{1/m} 
	\le C_u (c_{n,m}u\wed \ol u h)^{1/m},
$$
where $C_u$ is a positive constant depends only on $u$.
The last term is in $L^1_{loc}$ for $u\in V_m$.

In particular, if $\CI(h^{1/m})=\CO_X$, then
$V_m = H^0(X,mK_{X}+L)$ and hence
$V_m = H^0(X,(K_{X}+L_{m-1})\ot \CI(h_{m-1}))=H^0(X,mK_{X}+L)$.

(5)
In the case $V_{m}=0$ at the beginning, we formally agree that
$B_{m}\equiv 0,h^{(m)} \equiv +\infty, h_{m-1}\equiv+\infty$ (except for $h_{1NS}=h$).
\qed
\end{no}


\subsection{The $L^{2/m}$ metric, relative version}

\begin{setup}\label{relative}
Let $X$ and $Y$ be complex manifolds of $\dim X-\dim Y=n$, and let $f:X\to Y$ be a projective surjective morphism with connected fibers.
Let $L$ be a holomorphic line bundle on $X$ with a singular Hermitian metric $h$ with semi-positive curvature.
We let
$X_y$ the scheme theoretic fiber of $y\in Y$,
$L_y:=L|_{X_y}, h_{y}:=h|_{X_{y}}$ the restrictions (it can be $h_{y}\equiv +\infty$),
$$
	Y_0 \subset Y
$$
the maximum Zariski open subset where $f$ is smooth, $\Dl:=Y\sm Y_{0}$, and let
$$
	Y_{h} =\{y\in Y_0; h_{y}\not\equiv +\infty \}. 
$$
For every integer $m\ge 1$ and $y\in Y_{h}$, we can discuss the Narashimhan-Simha construction on the fiber $X_{y}$ with $(L_{y},h_{y})$ as in \ref{NSc}, and have fiberwise objects $V_{m,y}, B_{m,y}, h^{(m)}_{y}, h_{m-1,y}, g_{m,y}=g_{mNS,y}$ accordingly (it can be $V_{m,y}=0$).
\qed
\end{setup}

We remark that the set $Y_{h}$ is Zariski dense in $Y$, however may not be Zariski open.
The complement $Y \sm Y_{h}$ is a pluripolar set and hence can be quite different from algebraic/analytic objects.
For example, a pluripolar set may not be closed in Hausdorff topology, and can be Zariski dense. 
In any case we adopt the conventions that if $h_{y}\equiv +\infty$ for a point $y\in Y_{0}$ i.e., $y\in Y_{0}\sm Y_{h}$, then $V_{m,y}=0, B_{m,y}\equiv 0, h_{m-1,y}\equiv+\infty$ and so on.

\begin{rem}\label{L2metric} 
It is very important to understand the local expression of the metrics we consider above, as they will be needed in order to deduce their extension properties.

We give here an explicit formulation of the {\it (relative) canonical $L^{2}$-metric} $g_{1NS,X/Y}$ ($g_{1,X/Y}$ for short slightly) on $f_{\star}(K_{X/Y}+L)|_{Y_{0}}$ with respect to $h$ in \ref{relative}.
For this purpose, we may suppose that $Y$ itself is a coordinate neighborhood.

(1)
Let $\eta \in H^{0}(Y, K_{Y})$ be a nowhere vanishing section, trivializing the canonical bundle of $Y$; in particular we have $K_{Y}=\CO_{Y}\eta$.
We recall that we have 
$$H^{0}(Y,f_{\star}(K_{X/Y}+L))=H^{0}(Y, \mathcal Hom\, (K_{Y},f_{\star}(K_{X}+L))),$$ and therefore every section $u \in H^{0}(Y,f_{\star}(K_{X/Y}+L))$ corresponds to an $\CO_{Y}$-homomorphism 
$u:K_{Y} \to f_{\star}(K_{X}+L)$.
We still use the same symbol $u$ for the induced homomorphism $u : H^{0}(Y,K_{Y}) \to H^{0}(Y,f_{\star}(K_{X}+L))=H^{0}(X,K_{X}+L)$.
Thus we can write $u(\eta) \in H^{0}(X,K_{X}+L)$.
Let $\{U_{\lam} \}_{\lam}$ be a local coordinate system of $X$.
Regarding $u(\eta)$ as an $L$-valued top-degree holomorphic form on each $U_{\lam}$,
there exists $\sg_{u\lam} \in  H^{0}(U_{\lam}\sm f^{-1}(\Dl), \Omega_{X}^{n}\ot L)$
such that we have  
$$
	u(\eta)=\sg_{u\lam}\wed f^{\star}\eta
$$ 
on $U_{\lam}\sm f^{-1}(\Dl)$ i.e.\ we can ``divide" $u(\eta)$ by $f^{\star}\eta$ where $f^{\star}\eta$ has no zeros.
The choice of $\sg_{u\lam}$ is not unique (the ambiguity lies in the image of $\Omega_{X}^{n-1}\ot f^{\star}\Omega_{Y}^{1}$), however the restriction $\sg_{u\lam}|_{X_{y}} \in H^{0}(U_{\lam}\cap X_{y}, K_{X_{y}}+ L_{y})$ on each smooth fiber $X_{y} \ (y\in Y_{0}$) is unique and independent of the local frame $\eta$.
Finally the collection $\{\sg_{u\lam} \}_{\lam}$, resp.\ $\{\sg_{u\lam}|_{X_{y}} \}_{\lam}$, glue together as a global section 
$$
	\sg_{u} \in H^{0}(X \sm f^{-1}(\Dl), K_{X/Y}+ L), 
	\quad \sg_{uy} \in H^{0}(X_{y}, K_{X_{y}}+ L_{y})
$$ 
respectively; moreover, the later is the restriction of the former  $\sg_{uy}= \sg_{u}|_{X_{y}}$ for $y \in Y_{0}$.
This construction independent of the reference metrics.

(2)
Then the canonical $L^{2}$-metric $g_{1NS,X/Y}$ for $u, v \in H^{0}(Y,f_{\star}(K_{X/Y}+L))$ at $y \in Y_{0}=Y\sm \Dl$ is given by the expression
$$
	g_{1NS,X/Y}(u,v)(y) = \int_{X_{y}} c_{n} \sg_{u}|_{X_{y}} \wed \ol \sg_{v}|_{X_{y}} h_{y}.
$$
We may write the integrand as $(c_{n} \sg_{u} \wed \ol \sg_{v} h)|_{X_{y}}$ by an abuse of notations, where $c_{n}=(-1)^{n^{2}/2}$.
The convergence of the last integral depends on $u$ and $v$.
There may be other ways to express $g_{1NS,X/Y}$, for example $g_{1NS,X/Y}=\{g_{1NS,X/Y,y}\}_{y\in Y_{0}}$ as a family of fiberwise Hermitian forms, where $g_{1NS,X/Y,y}$ is the restriction of $g_{1NS,y}$ on the image of $f_{\star}(K_{X/Y}+L)_{y}$ in $H^{0}(X_{y}, K_{X_{y}}+L_{y})$ (the latter space can be larger in general), and where $g_{1NS,y}$ is the canonical $L^{2}$-metric on $H^{0}(X_{y}, K_{X_{y}}+L_{y})$ with respect to $h_{y}$ (it can be $g_{1NS,y}(u,u)=+\infty$ for some $u \in H^{0}(X_{y}, K_{X_{y}}+L_{y})$).
\qed
\end{rem}

\vskip2mm

\begin{rem}\label{restriction} 
As a consequence of Ohsawa-Takegoshi's extension theorem, we have the inclusion $\CI(h_{y}) \subset \CI(h) \cdot \CO_{X_{y}}$ for any $y\in Y_{0}$. 
The next set we will be interested in would be $y\in Y_{0}$ such that 
$$
	\CI(h_{y}) = \CI(h) \cdot \CO_{X_{y}}.
$$
In the algebraic case, this holds for any $y$ in a Zariski open subset (\cite[9.5.35]{PAG}).
We show here that this equality holds on a set of full measure on $Y$ (which in general is not Zariski open).
To this end, it is enough to show that $\CI(h) \cdot \CO_{X_{y}} \subset \CI(h_{y})$ for a.a.\,$y\in Y$.
We may suppose that $Y$ is a small coordinate neighborhood.
Let $U\subset X$ be a local coordinate set, which is isomorphic to a polydisk, such that $f|_{U}$ is 
(conjugate to) the projection to a sub-polydisk.
We may assume that $\CI(h)|_{U}$ is generated by a finite number of holomorphic functions $s_{1},\ldots, s_{k}\in H^{0}(U,\CO_{X})$, in particular $|s_{i}|^{2}h \in L^{1}_{loc}(U)$.
As these $s_{i}|_{X_{y}}$ generate $\CI(h) \cdot \CO_{X_{y}\cap U}$, it is enough to show that each $s_{i}|_{X_{y}} \in \CI(h_{y})|_{X_{y}\cap U}$ for a.a.\,$y\in f(U)$.
By Fubini's theorem, $|s_{i}|_{X_{y}}|^{2}h_{y} \in L^{1}_{loc}(X_{y}\cap U)$ for a.a.\,$y\in f(U)$, i.e., $s_{i}|_{X_{y}} \in \CI(h_{y})|_{X_{y}\cap U}$ for a.a.\,$y\in f(U)$.
\smallskip
\end{rem}

\begin{rem}\label{modif}
We note that $f_{\star}(K_{X/Y}+ L)$ on $Y$ and the canonical $L^{2}$-metric $g_{1}=g_{1NS,X/Y}$ on $Y_{0}$ are invariant under  bimeromorphic transform of $X$.
Let $\mu : X' \to X$ be a proper bimeromorphic morphism from a complex manifold $X'$, $f'=f\circ \mu :X' \to Y$, and let $L'=\mu^{\star}L, h'=\mu^{\star}h$.
Then it is easy to see $f'_{\star}(K_{X'/Y}+ L') =f_{\star}(K_{X/Y}+ L)$.
Suppose that $f':X'\to Y$ is smooth over $Y_{0}'\subset Y$, and let $g_{1}' =g_{1NS,X'/Y}$ be the canonical $L^{2}$-metric with respect to $h'$ on $f'_{\star}(K_{X'/Y}+ L')|_{Y_{0}'}$.
Then $g_{1}=g_{1}'$ on $Y_{0}\cap Y_{0}'$ (one should observe this in the case $Y$ is a point, and then the relative case).
As a result $g_{1}$ and $g_{1}'$ define a singular Hermitian metric of $f_{\star}(K_{X/Y}+ L)|_{Y_{0}\cup Y_{0}'}$.
(We do not know if $Y_{0}'\subset Y_{0}$ holds.) 
\end{rem}

\medskip


\subsection{Direct image sheaves of adjoint type and their metric properties} 

In this subsection we show that the direct image sheaves $f_{\star}(K_{X/Y}+ L)$ can be endowed with a singular Hermitian metric with semi-positive curvature in the sense we have defined in \ref{sHm2}.
We remark that in general, the singularities of the metric we construct cannot be avoided. The main result we obtain here
can be seen as a more general and precise form of \cite[3.5]{BPDuke}, and it implies various notions of weak positivity in algebraic geometry. The set-up and notations in what follows are the same as in \ref{relative}.

\begin{no}\label{set on Y}
To start with, we define the following sets, which play an important role for the understanding of the metrics we will define next.
We recall that $Y_0\subset Y$ is the set of regular values of $f$, and that 
$Y_h=\{ y\in Y_0;\ h_{y}=h|_{X_{y}}\not\equiv +\infty\}$.
\begin{equation*} 
\begin{aligned}
Y_{1,h,\rm ext} & :=\{y\in Y_{h};\ 
		H^{0}(X_{y},(K_{X_{y}}+L_{y})\ot \CI(h_y)) = H^{0}(X_{y},K_{X_{y}}+L_{y}) \}, 
\\
Y_{1,\rm ext} & := \{y \in Y_{0};\ 
		h^0(X_y,K_{X_y}+L_y) \text{ equals to the rank of } f_{\star}(K_{X/Y}+L) \}.
\\
Y_{1,\rm free} & := \text {The largest Zariski open subset such that } f_\star(K_{X/Y}+L) \text{ is locally free. }	\end{aligned}
\end{equation*} 
\end{no}

\noindent 
We remark that the set $Y_{1,\rm ext}$ and $Y_{1,\rm free}$ are independent of the metric $h$ and $Y_{1,\rm ext}\subset Y_{1,\rm free}$. 
We have then the following observation.

\begin{lem}\label{Y1h}

\noindent {\rm (1)} We have the inclusion $Y_{1,h,\rm ext} \subset Y_{1,\rm ext}\cap Y_h$.
\smallskip

\noindent {\rm (2)} Let $y \in Y_{1,h,\rm ext}$.
Then the equalities 
$$
f_{\star}(K_{X/Y}+L)_{y}
=H^{0}(X_{y},K_{X_{y}}+L_{y})
=H^{0}(X_{y},(K_{X_{y}}+L_{y})\ot \CI(h_y))
$$
hold. 
As a consequence of {\rm \ref{L2metric}(2)}, the restriction of the direct image $f_{\star}(K_{X/Y}+L)$ to the set $\displaystyle Y_{1,h,\rm ext}$ can be endowed with the canonical $L^{2}$-metric $g_{1NS,y}$ with respect to $h_{y}$.
Moreover $f_\star(K_{X/Y}+L)$ is locally free at $y$, and the natural inclusion $f_{\star}((K_{X/Y}+L)\ot \CI(h)) \subset f_{\star}(K_{X/Y}+L)$ is isomorphic at $y$. Moreover, the inclusion 
$Y_{1,h,\rm ext} \subset Y_{1,\rm ext}$ holds.

\smallskip

\noindent {\rm (3)} If the natural inclusion $f_{\star}((K_{X/Y}+L)\ot \CI(h)) \subset f_{\star}(K_{X/Y}+L)$ is generically isomorphic, then $Y_{1,h,\rm ext}$ is not empty and $Y\sm Y_{1,h,\rm ext}$ has measure zero. 
\smallskip

\noindent {\rm (4)}  
In conclusion, the sheaf $\displaystyle f_{\star}(K_{X/Y}+L)|_{Y_{1,\rm ext}}$ can be endowed with the canonical $L^{2}$-metric $\displaystyle \{g_{1NS,y}\}_{y\in Y_{1,\rm ext}}$; however, we note that 
we have $\det g_{1NS,y}=+\infty$ if $y \in Y_{1,\rm ext}\sm Y_{1,h,\rm ext}$.
\end{lem}

\begin{proof} The first point (1) will be established at the end of the argument for we give next for (2). 
\smallskip
 
\noindent (2)
By Ohsawa-Takegoshi, every section $u\in H^0(X_y,K_{X_y}+L_y)$ admits an extension 
$$\wtil u \in  H^{0}(X_W,(K_{X/Y}+L)\ot \CI(h))$$ 
defined on some neighborhood $X_W$ of $X_y$.
In particular the natural induced map 
$$f_{\star}\big((K_{X/Y}+L)\ot \CI(h)\big)_{y} \to H^0(X_y,K_{X_y}+L_y)$$ is surjective
(we note that the left-hand side direct image above is contained in 
$ f_{\star}(K_{X/Y}+L)_{y}$).
The cohomology base change theorem implies that $f_*(K_{X/Y}+L)$ is locally free at $y$, and 
as a consequence the natural inclusion $f_{\star}((K_{X/Y}+L)\ot \CI(h)) \subset f_{\star}(K_{X/Y}+L)$ is isomorphic at
$y$ .
This argument also shows (1).
\smallskip
 
\noindent (3)
There exists a non-empty Zariski open subset $W\subset Y$ such that the inclusion map
$$f_{\star}((K_{X/Y}+L)\ot \CI(h)) \subset f_{\star}(K_{X/Y}+L)$$ 
is isomorphic on $W$, and such that 
$$
f_{\star}((K_{X/Y}+L)\ot \CI(h))_{y} = H^{0}(X_{y},(K_{X_{y}}+L_{y})\ot \CI(h)\cdot \CO_{X_{y}}),  
f_{\star}(K_{X/Y}+L)_{y}= H^{0}(X_{y},K_{X_{y}}+L_{y})
$$ 
for any $y\in W$.
Since $\CI(h)\cdot\CO_{X_{y}}=\CI(h_{y})$ for a.a.\,$y\in Y_{0}$ in general by \ref{restriction}, so our assertion follows.
\smallskip
 
\noindent (4) This point follows directly from the preceding ones.
\end{proof}
\medskip

\noindent We recall next the following result.

\begin{thm}\cite[3.5]{BPDuke} 
Let $f:X\to Y$ and $(L,h)$ be as in {\rm\ref{relative}}.
Suppose that $f$ is smooth and $f_{\star}(K_{X/Y}+L)$ is locally free.
Then the fiberwise canonical $L^{2}$-metrics 
$$
	g_{1NS,y}(u,u):=\int_{X_{y}} c_{n}u\wed\ol u h_{y} \le +\infty
$$ 
for $u \in H^{0}(X_{y},K_{X_{y}}+L_{y})$ induces a singular Hermitian metric $g_{1NS,X/Y}$ on $f_{\star}(K_{X/Y}+L)$ with positive curvature.
\end{thm}
\smallskip 

\noindent The notion of \emph{singular Hermitian metric} $h$ on a vector bundle $E$  adopted in \cite[p.\,357]{BPDuke} is more general than \ref{sHm2}: one only assume that $h$ corresponds locally to a measurable map from the base manifold to the space of semi-positive definite Hermitian matrices
(hence, no assumption concerning the determinant metric).
The negativity of its curvature is defined by the plurisubharmonicity of
$\log |v|_{h}^{2}$ for any local section $v\in H^{0}(U,E)$.
We note that stated in this way, the notion of singular Hermitian metric makes not much sense, since any vector bundle admits $h\equiv 0$ as a singular Hermitian metric with negative curvature.

\noindent In our set-up in this paper, the previous result has the following consequence.

\begin{thm}\cite[3.5]{BPDuke}\label{bp35}
Let $f:X\to Y$ and $(L,h)$ be as in {\rm\ref{relative}}.
Suppose that $f$ is smooth, $f_{\star}(K_{X/Y}+L)$ is locally free, and that the natural inclusion 
$$
	f_{\star}((K_{X/Y}+L)\ot \CI(h)) \subset f_{\star}(K_{X/Y}+L)
$$ 
is generically isomorphic.
Then $f_{\star}(K_{X/Y}+L)$ admits a singular Hermitian metric $g_{1NS,X/Y}=\{g_{1NS,X/Y,y}\}_{y\in Y}$ with positive curvature and satisfying a {\rm base change property} on $Y_{1,\rm ext}$, which means, $g_{1NS,X/Y,y}=g_{1NS,y}$ holds on every fiber $f_{\star}(K_{X/Y}+L)_{y}$ at $y\in Y_{1,\rm ext}$, where $g_{1NS,y}$ is the canonical $L^{2}$-metric on $f_{\star}(K_{X/Y}+L)_{y} = H^{0}(X_{y},K_{X_{y}}+L_{y})$ with respect to $h_{y}$ in 
{\rm \ref{Y1h}(4)}.
\end{thm}

\noindent The statement concerning the equality $g_{1NS,X/Y,y}= g_{1NS,y}$ on $f_{\star}(K_{X/Y}+L)_{y}$ for $y\in Y_{1,\rm ext}$ is slightly more informative than the original one in \cite[3.5]{BPDuke}.
We stress on the fact that the statement above is implicit in \cite{BPDuke}, it is not mentioned explicitly.

\medskip

\noindent Our version of the previous result reads as follows.

\begin{thm}\label{ext} 
Let $f:X\to Y$ and $(L,h)$ be as in {\rm\ref{relative}}.
We suppose that the natural inclusion $f_{\star}((K_{X/Y}+L)\ot \CI(h)) \subset f_{\star}(K_{X/Y}+L)$ is generically isomorphic.
Then the canonical $L^{2}$-metric $g_{1NS,X/Y}$ on $f_{\star}(K_{X/Y}+L)|_{Y_0\cap Y_{1,\rm free}}$ with respect to $h$ has positive curvature by {\rm \ref{bp35}}, and it extends as a singular Hermitian metric $\wtil g_{1NS,X/Y}$ on the torsion free sheaf $f_{\star}(K_{X/Y}+L)$ with positive curvature.
In particular $\wtil g_{1NS,X/Y}$ has a base change property on $Y_{1,\rm ext}$.
\end{thm}

The rest of this subsection is devoted to the proof of \ref{ext}.
We start by fixing some notations, and making some standard reductions.
Let 
$$
	E:=f_{\star}(K_{X/Y}+L), \ \ \  \text{resp.\ $g:=g_{1NS,X/Y}$},
$$
be the direct image of the bundle $K_{X/Y}+L$, resp.\ the metric on $E|_{Y_0\cap Y_{1,\rm free}}$.
Noting that we have $\codim (Y\sm Y_{1,\rm free})\ge 2$, it is enough to show that $g$ extends to a singular Hermitian metric on $E|_{Y_{1,\rm free}}$ with
positive curvature by definition of curvature positivity of singular Hermitian metric of torsion free sheaves (see \ref{dual}).
Thus we may assume from the start that $E$ is locally free, i.e., $Y_{1,\rm free}=Y$.
We also note that the extension of $g$ is  a local question on $Y$.
Moreover, we can freely restrict ourselves on a Zariski 
open subset $Y'' \subset Y$ with $\codim (Y\sm Y'')\ge 2$, since the extended metric is unique.

In conclusion, it is enough to consider the following local setting, near the general point of 
a codimension 1 subset $\Dl \subset Y$ (possibly after taking a modification of $X$ along $f^{-1}(\Del)$; such operations do not affect $E$ and the metric $g$, see \ref{modif}).

\begin{setup} \label{local}
Let $f : X \to Y$ and $(L, h)$ be as in \ref{relative}.
In addition to the assumption in \ref{ext}, let us assume further the following

\smallskip
\noindent
(1)
The base $Y$ is a unit polydisk in $\BC^m$ with coordinates $t = (t_1, \ldots, t_m)$.
Let $dt = dt_1 \wed \ldots \wed dt_m \in H^0(Y, K_Y)$ be a global frame of $K_{Y}$.

(1.i) $f$ is flat, and its discriminant locus $\Dl \subset Y$ is given by the equation $\Dl = \{t_m = 0\}$ (or $\Dl = \emptyset$), $Y_{0}=Y\sm \Dl$,

(1.ii) the effective divisor $f^*\Dl$ has a normal crossing support, and

(1.iii) the morphism $\Supp f^*\Dl \to \Dl$ is relative normal crossing (see \ref{rnc}).

\smallskip
\noindent
(2)
(2.i) $E=f_{\star}(K_{X/Y}+L) \cong \CO_Y^{\oplus r}$, i.e., globally free and trivialized of rank $r > 0$.

(2.ii)@For every point $y\in Y$, the naturally induced homomorphism $E/(\fm_{Y,y}E) \to H^{0}(X_{y},$ $(K_{X/Y}+L)\ot \CO_{X_{y}})$ is injective, where $\fm_{Y,y} \subset \CO_{Y}$ is the maximal ideal and $X_{y}$ is the scheme theoretic fiber of $y$; $X_{y}$ is defined by the ideal $\fm_{Y,y}\CO_{X}$ (see \ref{inj}).

\smallskip
\noindent
(3)
Let $g$ be the canonical $L^{2}$-metric on $E$ over $Y_{0}$, which is positively curved by \ref{bp35}.

\smallskip
\noindent
We may replace $Y$ by slightly smaller polydisks, or may assume everything is defined over a slightly larger polydisk. 
\qed
\end{setup}

\begin{dfn}\label{rnc}
(1) In the set-up \ref{local} above, by \emph{$\Supp f^*\Dl \to \Dl$ is relative normal crossing} we  mean that locally near every point $x \in X$, there exists a local coordinate $(U; z = (z_1, \ldots, z_{n+m}))$ such that $f|_U$ is given by $t_1 = z_{n+1}, \ldots, t_{m-1} = z_{n+m-1}$, $t_m = z_{n+m}^{b_{n+m}} \prod_{j=1}^{n} z_j^{b_j}$ with non-negative integers $b_j$ and $b_{n+m}$.
\smallskip

\noindent (2) The map $f$ is said to be {\it semi-stable}, if moreover $f^{\star}\Dl$ is reduced (in the notations above, this simply means that $b_{j}\leq 1$).
\qed
\end{dfn}

We explain \ref{local}(2.ii) in a slightly general context.

\begin{lem}\label{inj}
Let $f:X\to Y$ be a proper surjective morphism of complex manifolds, 
$F$ a torsion free sheaf on $X$ which is flat over $Y$, and let $E=f_{*}F$ be its direct image, which is merely torsion free.
Then there exists a Zariski closed subset $Z\subset Y$ with $\codim_{Y}Z \ge 2$ such that $E|_{Y\sm Z}$ is locally free, and such that
for every point $y\in Y \sm Z$, the natural homomorphism $E/(\fm_{Y,y}E) \to H^{0}(X_{y},F_y)$ is injective, where $F|_{X_{y}}=F\ot \CO_{X_{y}}$.
\end{lem}

\begin{proof}
Our assertion is local on $Y$ around complement of some codimension 2 Zariski closed subsets of $Y$. 
Let $m=\dim Y\ge 1$.
As $E$ is torsion free, we may suppose $E$ is locally free on $Y$ by removing codimension 2 subsets in $Y$. 
By general results (such as the upper-semi-continuity theorem, and the cohomology and base change theorem), there exist Zariski closed subsets $Z\subset Y$ with $\codim_{Y}Z \ge 1$ such that for every point $y\in Y \sm Z$, the natural map $E/(\fm_{Y,y}E) \to H^{0}(X_{y},F_y)$ is isomorphic.
By removing codimension 2 subsets in $Y$ and replacing $Y$, we may suppose that $Z$ is smooth, irreducible and of $\dim Z=m-1$.
In fact, we may suppose $Z=\{s_{1}=0\}$ in a coordinate of $Y \subset (\BC^{m},(s_{1}, \ldots, s_{m}))$.
Denote by $f_{1}: X_{1}\to Z$ be the induced morphism, where $X_{1}$ is the scheme theoretic inverse image of $Z$, namely in this case $X_{1}$ is a divisor in $X$ defined by $\{f^{*}s_{1}=0\}$.
In general $F\ot \CO_{X_{1}}$ is flat over $Z$, and $E_{1}:=(f_{1})_{*}(F\ot \CO_{X_{1}})$ is a coherent sheaf on $Z$.
If $Y$ is a curve, then $Z$ is a point $y \in Y$ and $E_{1}=H^{0}(X_{y},F_y)$.

We have an exact sequence $0\to F \ot \CO_{X}(-X_{1}) \to F \to F \ot \CO_{X_{1}} \to 0$ on $X$, and an exact sequence
 $0\to f_{*}(F \ot \CO_{X}(-X_{1})) \to f_{*}F \to f_{*}(F \ot \CO_{X_{1}})$ on $Y$.
Noting that $\CO_{X}(-X_{1})=f^{*} \CO_{Y}(-Z)$ is an invertible sheaf, the last exact sequence turns to be
$0\to \CO_{Y}(-Z)\ot E \to E \to E_{1}$.
Thus the induced homomorphism $E/(s_{1}E) \to E_{1}$ is injective.
If $Y$ is a curve, we are done at this point.
We suppose $m\ge 2$.

By removing codimension 1 subsets in $Z$ and replacing $Y$, we may suppose that $E_{1}$ is locally free on $Z$, more strongly  for every point $y\in Z$, the natural map $E_{1}/(\fm_{Z,y}E_{1}) \to H^{0}(X_{y},F_y)$ is isomorphic, and that the map $E':=E/(s_{1}E) \to E_{1}$ is injective as a vector bundle homomorphism (not only as a sheaf homomorphism).
In particular $E'$ can be seen as a direct summand of $E_{1}$ (as we are in local situation).
In general the rank of $E'$ would be strictly smaller than that of $E_{1}$.
For every point $y\in Z$, we thus have an injection $E'/(\fm_{Z,y}E') \to H^{0}(X_{y},F_y)$ as a composition.
As $E'/(\fm_{Z,y} E') =E/(\fm_{Y,y}E)$ (this can be checked by taking a local coordinate), our claim is proved.
\end{proof}

\begin{lem}
In order to prove Theorem {\rm \ref{ext}}, we can further suppose that $f$ is semi-stable in addition to {\rm \ref{local}}.
\end{lem}

\begin{proof}
We use the weakly semi-stable reduction (or the semi-stable reduction in codimension 1)
(\cite[Ch.\ II]{KKMS}\,\cite[\S 7.2]{KM}\,\cite[\S 6.4]{Vbook}).
Let 
$$
	f^*\Dl = \sum\nolimits_j b_j B_j
$$ 
be the prime decomposition.
Let $Y'$ be another copy of a unit polydisk in $\BC^m$ with coordinates
$t' = (t_1', \ldots, t_{m-1}', t_m')$. 
Let $\ell$ be the least common multiple of all $b_j$.
Let $\tau : Y' \to Y$ be a ramified covering given by
$(t_1', \ldots, t_{m-1}', t_m') \mapsto (t_1', \ldots, t_{m-1}', (t'_m)^\ell)$,
and $X^{\circ} = X \times_Y Y'$ be the fiber product.
Let $\nu : X' \to X^{\circ}$ be the normalization, and 
$\mu : X'' \to X'$ be a resolution of singularities,
which is biholomorphic on the smooth locus of $X'$.
We encode next the induced morphisms in the following diagram.
\begin{equation*} 
\begin{CD}
    X''  @>{\mu}>>  X' @>{\nu}>> X^{\circ}= X \times_Y Y' @>{\tau^{\circ}}>> X \\
    @Vf''VV           @Vf'VV           @V{f^{\circ}}VV              @VVfV  \\ 
    Y'   @=   Y'   @=  Y'           @>>{\tau}>   Y
\end{CD}
\end{equation*} 
We let also $\tau'' : X'' \to X$ be the composition, $L'' = {\tau''}^*L$,
and $h''= {\tau''}^*h$ the induced singular Hermitian metric on $L''$.
We set $\Dl' =\Supp \tau^{\star}\Dl=\{t_{m}'=0\}\subset Y'$.
The fact that we can take a resolution $\mu : X'' \to X'$ so that $f'': X'' \to Y'$ is semi-stable in codimension 1 is by no means ``obvious", but it is established in the references quoted at the beginning of the proof.

By \cite[Lemma 3.2]{Vi1}\,\cite[V.3.30]{Nbook} (we note that adding $L'' = {\tau''}^*L$ is harmless), we then have a natural inclusion homomorphism
$$
  \ga : f''_{\star}(K_{X''/Y'}+ L'') \lra \tau^*f_{\star}(K_{X/Y}+L),
$$
which is isomorphic over $Y' \sm \Dl'$.
In particular $f''_{\star}(K_{X''/Y'}+ L'')$ is locally free over $Y' \sm \Dl'$.
The natural inclusion $f''_{\star}((K_{X''/Y'}+ L'')\ot \CI(h'')) \subset f''_{\star}(K_{X''/Y'}+ L'')$ is also generically isomorphic, and $E'':=f''_{\star}(K_{X''/Y'}+ L'')$ admits also the canonical $L^{2}$-metric $g'=g_{1NS,X''/Y'}$ over $Y' \sm \Dl'$ with respect to $h''$ (by \ref{bp35}).
Moreover the inclusion $\ga$ induces an isometry over $Y' \sm \Dl'$, i.e.,
$g' = \ga^{\star}\tau^{\star}g$ as a singular Hermitian metric on $E''|_{Y' \sm \Dl'}$.
In fact, since $\tau : Y'\sm \Dl' \to Y \sm \Dl$ is locally biholomorphic, the coordinate free description of the $L^{2}$-metric in \ref{L2metric} explain the isometry.

By \ref{pullback}, it is enough to prove that the singular Hermitian metric $\tau^{\star}g$ over $Y'\sm \Dl'$ can be extend as a singular Hermitian metric on $\tau^{\star}E= \tau^*f_{\star}(K_{X/Y}+L)$ over $Y'$ with positive curvature.
To this end, we use  \ref{subquot2}(3), which shows that it is enough to prove that the singular Hermitian metric $g'$ over $Y'\sm \Dl'$ can be extend as a singular Hermitian metric on $E''=f''_{\star}(K_{X''/Y'}+ L'')$ over $Y'$ with positive curvature. 
 Once this is established, we apply  \ref{subquot2}(3), which shows that the extension of $g'$ with positive curvature induces a singular Hermitian metric $g^{+}$ on $\tau^{\star}E= \tau^*f_{\star}(K_{X/Y}+L)$ with positive curvature.
The metric $g^{+}$ extends the pull-back $\tau^{\star}g$, because of $g'= \ga^{\star}\tau^{\star}g$ on $Y' \sm \Dl'$ (and because of the proof or construction of  \ref{subquot2}(3)).
And finally, in order to extend $g'$ over $Y'$
(by removing codimension two subvarieties of $Y'$ and localizing again), we may suppose that $f''$ is semi-stable and $E''$ satisfies \ref{local}(2).
\end{proof}

\begin{prop} \label{bdd}
Theorem {\rm \ref{ext}} holds true when $f$ is semi-stable (hence in general, by the result above).
\end{prop}

\begin{proof} In what follows, we will assume that the relations (1)--(3) in 
\ref{local} are verified, and that the map $f$ is semi-stable. This will allow us to perform a few explicit computations in order to understand the behavior of the metric near $\Dl$.

By \ref{bp35}, the metric $g^{\star}$ is negatively curved on $Y_{0}$.
Our goal is to extend it on $Y$ with negative curvature.
We consider any open set $U\subset Y$ and a non-zero $\xi \in H^{0}(U, E^{\star})$.
We may assume that $U=Y$;
then by hypothesis $\log |\xi |_{g^{\star}}^{2}$ is psh on $Y_{0}$.
The conclusion of {\rm \ref{ext}} is that the function $\log |\xi |_{g^{\star}}^{2}$ 
extends as a psh function on $Y$.
As it is well known (by Hartogs theorem) the extension property holds 
provided that we can show that there exists a constant $c_{\xi}>0$ such that
$|\xi |_{g^{\star}}^{2}< c_{\xi}$ on $Y_{0}$ (possibly after shrinking $Y$).

We consider a global frame $e_1, \ldots, e_r \in H^0(Y, E)$ and we trivialize $E=Y\times \BC^{r}$ via this frame.
With respect to this trivialization, the metric $g$ can be written as a matrix valued function $g=(g_{i\ol j})$ on $Y_{0}$.
For every $t\in Y_{0}$, we denote by
$0<  \lam_{1}(t) \le \lam_{2}(t) \le \ldots \le  \lam_{r}(t) \le +\infty$ the eigenvalues of $g(t)=(g_{i\ol j}(t))$, see \ref{eigenbd}.
Then we have $\lam_{r}(t) <+\infty$ for every $t \in Y_{1,h,\rm ext}$ by definition \ref{set on Y} of $h_{1,h,\rm ext}$.
We shall show in \ref{Ft12} that there exists a constant $c_{0}>0$ independent of $t\in Y_{0}$ such that $\lam_{1}(t)> c_{0}$ for all $t\in Y_{0}$ (including the case $\lam_{1}(t)=+\infty$).
This implies the existence of a constant $c_{\xi}>0$ such that
$|\xi |_{g^{\star}}^{2}< c_{\xi}$ on $Y_{0}$, and our claim will be proved.
\smallskip

For a constant vector $s = (s_1, \ldots, s_r) \in \BC^r$, we let $u_s = \sum_{i=1}^r s_ie_i \in H^0(Y, E)=H^{0}(X,K_{X/Y}+L) $.
We denote by $S^{2r-1} = \{ s \in \BC^r ; \ |s| = (\sum |s_i|^2)^{1/2} = 1 \}$ the unit sphere.
Since $e_1, \ldots, e_r$ generate $E$ over $Y$, $u_s$ is nowhere vanishing on $Y$ as soon as $s \ne 0$, namely $u_s$ is non-zero in $E/(\mathfrak m_{Y, y} E)$ at any $y \in Y$.
The map $\BC^r \to H^0(X, K_{X}+L)$ given by 
$$
	s \mapsto u_s \mapsto u_{s}(dt) = \sum_{i=1}^r s_i e_{i}(dt), \ \ 
	\BC^{r} \to H^{0}(X,K_{X/Y}+L) \cong H^{0}(X,K_{X}+L)
$$  
is continuous, with respect to the standard topology of $\BC^r$ and the topology of $H^0(X, K_{X}+L)$ of uniform convergence on compact sets
(recall \ref{L2metric} for the meaning of $u_{s}(dt)$).
Then the following \ref{Ft12} shows our claim.
\end{proof}

\begin{lem} \label{Ft12}
{\rm (cf.\ {\rm \cite[1.12]{Ft}}.)} \ 
Let $y \in \Dl$. 
Then there exist a neighborhood $W_{y}$ of $y$ in $Y$ and a constant $c_{0}>0$, such that $g(u_s, u_s)(t) \ge c_{0}$ for any $s \in S^{2r-1}$ and any $t \in W_{y} \sm \Dl$.
\end{lem}

Since $S^{2r-1}$ is compact, \ref{Ft12} is a consequence of the next statement.

\begin{lem}  \label{Ft11}
{\rm (cf.\ \cite[1.11]{Ft}.)} \ 
In notations above, let $y \in \Dl$ and let $s_0 \in S^{2r-1}$. 
Then there exist a neighborhood $S(s_0)$ of $s_0$ in $S^{2r-1}$, a neighborhood $W_{y}$ of $y$ in $Y$ and a constant $c_{0}>0$ such that $g(u_s, u_s)(t) \ge c_{0}$ for any $s \in S(s_0)$ and any $t \in W_{y} \sm \Dl$.
\end{lem}

\begin{proof} 
(1)
Let $u \in H^0(Y, E)$, and assume $u$ does not vanish at $y$. 
Then we claim that there exists a  component $B_j$ in $f^*\Dl = \sum B_j$ such that $u(dt) \in H^0(X, K_{X}+L)$ does not vanish identically along $B_j \cap f^{-1}(y)$.

The isomorphism $H^{0}(X,K_{X/Y}+L) \cong H^{0}(X,K_{X}+L)$ is as $H^0(Y, \CO_{Y})$-modules.
In particular, this will show that $u(dt) \in H^0(X, K_{X}+L)$ does not vanish at $y \in Y$ as an element of an $H^0(Y, \CO_{Y})$-module.
We have to show that if $u(dt)$ vanishes identically along $f^{-1}(y)$, then $u$ vanishes at $y$, i.e., $u \in \mathfrak m_{Y,y}E$.
If $m = 1$, we can check it by a direct computation as in the next paragraph. 
If $m>1$, we cut $Y$ and $X$ by $\{t_{1} = 0\}$ (note $X \cap \{t_{1} = 0\}$ is still smooth)
and use the induction hypothesis.
More precisely, we let $Y_{1}=\{t_{1} = 0\} \subset Y, X_{1}=f^{-1}(Y_{1})$, and let $f_{1}:X_{1}\to Y_{1}$ be the induced morphism.
We have a restriction map
$r: H^0(X, K_{X}+L) \to H^0(X_{1}, (K_{X}+L)|_{X_{1}}) \cong H^0(X_{1}, K_{X_{1}}+L|_{X_{1}})$ by the adjunction $K_{X}|_{X_{1}}= K_{X_{1}}$.
The kernel is generated by $t_{1}$.
If $u(dt)$ vanishes identically along $f^{-1}(y)$, then by induction hypothesis, the image of $u(dt)$ in $H^0(X_{y}, (K_{X/Y}+L)|_{X_{y}})$ is zero.
Thus by the property \ref{local}(2), we have $u(dt) \in \fm_{Y,y}E$.

Let us suppose $m=1$, and let $t=t_{1}$ be the coordinate function on $Y$.
Suppose that $u(dt)$ vanishes identically along $f^{-1}(y)$, which is now a reduced divisor. 
We shall show that $u \in H^{0}(X,K_{X/Y}+L)$ is divided by $f^{-1}t$ as a holomorphic section, and obtain a contradiction. 
To see $u$ is divided by $f^{-1}t$, it is enough to show that $u$ vanishes identically along $f^{-1}(y)$ as a bundle valued continuous function, and it is enough to show that it vanishes identically along the smooth part of $f^{-1}(y)$ by the continuity.
So we take a smooth point $x$ of $f^{-1}(y)$ and a local coordinate $(U;z)$ around $x$ such that the map $f$ is given by
$$
	(z_1,\dots, z_{n+1})\mapsto t= z_{n+1},
$$
and $y=0$ in this coordinates $t$. 
With respect to the coordinates $(U; z)$, (not all, but) a representative for the element $u$ can be written as a $(n,0)$-form on $U$ as follows
$$
	u= \tau dz_1\wed\dots\wed dz_n, 
$$
where $\tau \in H^{0}(U, \CO_{X})$.
By this we mean that we have the equality
$$
	u(dt)= \tau dz_1\wed\dots\wed dz_n \wed f^\star(dt)
	=  \tau dz_1\wed\dots\wed dz_n \wed dz_{n+1}.
$$
as top degree forms defined on $U$.
The vanishing of $u(dt)$ along $f^{-1}(y) \cap U$ is equivalent to the vanishing of $\tau$, as well as of $u$, along $f^{-1}(y) \cap U$.


\smallskip

(2)
For our nowhere vanishing $u_{s_0}$, we take a component 
$$
	B = B_j
$$ 
in $f^*\Dl$ such that $u_{s_{0}}(dt)$ does not vanish identically along $B \cap f^{-1}(y)$.
We take a general point $x_0 \in B \cap f^{-1}(y)$, and a local coordinate $(V; z = (z_1, \ldots, z_{n+m}))$ centered at $x_0 \in X$ such that $f$ is given by $t=(t_{1},\ldots,t_{m})= f(z) = (z_{n+1}, \ldots, z_{n+m})$ on $V$.
In particular $(f^*\Dl)|_V = B|_V = \{z_{n+m} = 0\}$.
Over $V$, we may assume that $L$ and $h$ are also trivialized, i.e., $L|_V \cong V \times \BC$ and $h=e^{-\vph}$ with a (quasi-)psh function $\vph \in L^{1}_{loc}(V,\BR)$. 
In particular $e^{-\vph}> a$ on $V$ for a constant $a>0$ (possibly after shrinking $V$).

(3)
For every $s \in S^{2r-1}$, we have $u_s(dt) \in H^0(X, K_{X}+L)$.
Over $V$, it can be written as
$$
	u_{s}(dt)|_{V}=\tau_{sV} dz_{1} \wed \ldots \wed dz_{n+m}
	\text{ with } \tau_{sV}\in H^{0}(V,\CO_{X}),
$$
and hence $\tau_{sV} dz_{1} \wed \ldots \wed dz_{n+m} =
\big( \sg_{sV} dz_{1} \wed \ldots \wed dz_{n}\big) \wed f^{\star}dt$.
As in \ref{L2metric}, every $u_{s}\in H^{0}(Y,E)=H^{0}(Y, f_{\star}(K_{X/Y}+L))$ defines $\sg_{st} \in H^{0}(X_{y},K_{X_{y}} + L_{y})$ for every $t\in Y_{0}$.
Remark \ref{L2metric} also shows 
$u_{s}|_{V\cap X_{t}} = \big(\sg_{sV} dz_{1} \wed \ldots \wed dz_{n} \big)|_{V\cap X_{t}} $ on $V\cap X_{t}$.

At the point $s_0 \in S^{2r-1}$, since $u_{s_{0}}(dt)$ does not vanish identically along $B \cap f^{-1}(y)$ and since $x_0 \in B \cap f^{-1}(y)$ is general, we have $\sg_{s_{0}V}(x_0) \ne 0$.

(4)
By the continuity of $s \mapsto u_s \mapsto u_{s}(dt)$, we can take an $\ep$-polydisk 
$V(\ep) = \{z = (z_1, \ldots, z_{n+m}) \in V ; \ |z_i| < \ep \text{ for any } 1 \le i \le n+m \}$ centered at $x_0$ for some $\ep > 0$ and a neighborhood $S(s_0)$ of $s_0$ in $S^{2r-1}$, such that
$$
	A := \inf\{ |\sg_{sV}(z)| ; \	 s \in S(s_0), \ z \in V(\ep) \} > 0.
$$
We set $W_{y} := f(V(\ep))$, which is an open neighborhood of
$y \in Y$, since $f$ is flat (in particular it is open).
Then for any $s \in S(s_0)$ and any $t \in W_{y} \sm \Dl$, 
we have
\begin{equation*} 
\begin{aligned}
\int_{X_{t}} (-1)^{n^{2}/2} \sg_{s} \wed \ol \sg_{s} h|_{X_{t}} & 
\ge a \int_{X_{t} \cap V} (-1)^{n^{2}/2}  \sg_{s} \wed \ol \sg_{s} |_{X_{t} \cap V} \\
{} & 
\ge a \int_{z \in X_{t} \cap V(\ep)} A^2 \ dV_n 
\ = \
a A^2 (2\pi \ep^2)^n,
\end{aligned}
\end{equation*} 
where $dV_n = \bigwedge_{i=1}^n \ai dz_i \wed d\ol{z_i}$ is an euclidean volume form on $\BC^n$.
\end{proof}

\begin{rem}
(1)
It may be possible to show $f_{\star}((K_{X/Y}+L)\ot \CI(h))$ is Griffiths semi-positive  in \ref{ext} without assuming that 
$f_{\star}((K_{X/Y}+L)\ot \CI(h)) \subset f_{\star}(K_{X/Y}+L)$ 
is generically isomorphic.
However there are some technical difficulties.
To explain this, suppose $Y$ is a disk, and $X_{0}=\sum m_{i}X_{i}$ is the singular fiber over $0\in Y$. 
There are difficulties to compare the ideal $\CI_{X_{0}}, \CI(h)$, and $\text{div} (u)$ for $u \in H^{0}(X,(K_{X/Y}+L)\ot \CI(h))$.
Embedded components of $\CO_{X}/\CI(h)$ are difficult to handle in general.
More subtle thing is, because $h$ may have non-algebraic/analytic singularities, we cannot reduce to the normal crossing situation even after a modification of the manifold.
\smallskip

\noindent (2)
In the proof of \ref{ext}, it is not enough to show that, on each line $\ell$ on $Y$ which is not contained in $\Delta$, $g_{1NS,X/Y}|_{\ell\sm \Del}$ is bounded as $g_{1NS,X/Y}|_{\ell\sm \Del} \ge c_{\ell}$ by a constant $c_{\ell}>0$.
We need a uniform bound $c_{\ell}\ge c>0$ independent of $\ell$.
We can not assume $\dim Y=1$ without special attentions.
\end{rem}

\newpage

\section{Positivity of relative pluricanonical bundles}\label{Srcb}

In this part of our article we will provide the necessary tools needed in order to generalize Theorem \ref{ext} in the context of pluricanonical bundles. 
Part of our motivation is that the metric associated to the bundle $f_\star(mK_{X/Y}+ L)$ for $m\geq 2$ has better regularity properties than the one on
$f_\star(K_{X/Y}+ L)$, and this is important for applications. 
The path we will follow is to write $mK_{X/Y}+ L= K_{X/Y}+ L_{m-1}$. 
We endow the bundle $L_{m-1}$ with the metric
$h_{m-1}:= (B_{m, X/Y}^{-1})^{(m-1)/m}h^{1/m}$, and then we apply the results of the preceding section with the data $f$ and $(L_{m-1}, h_{m-1})$.  
Therefore, it is fundamental to understand the relative Bergman type metric $B_{m, X/Y}$ (as the metric $h$ on $L$ is given).

To start with, we will consider the case $m= 1$. In the first part of this section we will revisit the construction of the Bergman-type metric on $K_{X/Y}+ L$.
The proof we present here is different from the original argument in \cite{BPDuke}. It relies on the version of the Ohsawa-Takegoshi extension theorem obtained by B\l ocki \cite{blocki} and Guan-Zhou \cite{GZ}, combined with an approach due to Tsuji. 
The construction of the metric $B_{m, X/Y}^{-1}$ on $mK_{X/Y}+ L$ for $m\geq 2$ is treated afterwards; to a large extent, it is a consequence of the case $m=1$. In both cases, the metric is constructed first on a Zariski open set which will be explicitly described here.  
We show that the curvature current corresponding to $B_{m, X/Y}^{-1}$ is semi-positive definite. The extension of this current to the manifold $X$ will be a consequence of the version of the Ohsawa-Takegoshi extension theorem we establish here.

The set-up in this section is as follows: let $f:X\to Y$ be a projective surjective morphism of complex manifolds with connected fibers, and let $L\to X$ be a line bundle, endowed with a singular Hermitian metric $h_L=e^{-\varphi_{L}}$ whose curvature current is positive. 


\subsection{The log-plurisubharmonicity of the Bergman kernel metric revisited}

``Ideally", the relative Bergman kernel metric $B_{1,X/Y}$ on $-(K_{X/Y}+L)$ would be defined as a collection of fiberwise Bergman kernel metrics $B_{1,y}$ with respect to $h_y:= h_L|_{X_y}$ for $X_{y}$ smooth (notations as in \ref{relative}).
However, it is not a-priori clear that the resulting metric $B_{1,X/Y}$ is even semi-continuous, mainly because 
the dimension of the space of sections of the bundle $K_{X_y}+L_y$ 
may not be constant as a function of $y$.

In order to overcome this phenomena, we first define explicitly $B_{1,X/Y}$ on the set
$$
Y_{\rm ext} = Y_{1,\rm ext} := \{y \in Y_{0};\ 
		h^0(X_y,K_{X_y}+L_y) \text{ equals to the rank of } f_{\star}(K_{X/Y}+L) \},
$$
which is the Zariski open set of regular values $y\in Y$ such that every section of
$\displaystyle K_{X_y}+ L_{y}$ extends to the $f$-inverse image of an open set containing $y$ (as we defined in \ref{set on Y}).
Let $$
	X_{\rm ext} = f^{-1}(Y_{\rm ext})
$$ 
so that in particular the map 
$f: X_{\rm ext} \to Y_{\rm ext}$ is a submersion. 
In \cite[pp.\,346--367]{BPDuke} (or \ref{relative} here) we have defined the metric $B_{1, X/Y}^{-1}$ on $\displaystyle (K_{X/Y}+ L)|_{f^{-1}(Y_{\rm ext}\cap Y_{h})}$. 
We assign the value $+\infty$ on the set $Y_{\rm ext}\setminus Y_h$, and in this way $B_{1, X/Y}^{-1}$ is well-defined on $X_{\rm ext}$.
\medskip

\noindent The following result was basically established in \cite{BPDuke}; we will nevertheless provide here an alternative argument which is based on 
B\l ocki's and Guan-Zhou's version of the Ohsawa-Takegoshi extension theorem with optimal constant.

\begin{thm}\label{m=1} \cite[0.1]{BPDuke}
We assume that there exists a point $y\in Y_h$ such that 
$$
	H^0\big(X_y, (K_{X_y}+ L_y)\ot \CI(h_{y})\big) \neq 0.
$$
Then the curvature current corresponding to the metric $B_{1, X/Y}^{-1}|{X_{\rm ext}}$ defined above is positive. 
Moreover, the local weights of this metric are upper semi-continuous, and uniformly bounded on $X\setminus X_{\rm ext}$. Hence, $B_{1, X/Y}^{-1}|_{X_{\rm ext}}$ 
extends to $X$ in a unique manner, and it endows the bundle $K_{X/Y}+ L$ with a metric whose curvature current is positive definite. 
\end{thm}

\begin{proof} 
A first observation is that it suffices to show that the weights $\varphi_{X/Y}^{(1)}$
of $B_{1, X/Y}^{-1}|_{X_{\rm ext}}$ are psh. 
Indeed, by the original argument in \cite{BPDuke}, the resulting metric extends to $X$, and it has the positivity property stated in the last part of Theorem \ref{m=1}. We remark that the extension property amounts to prove that $\displaystyle \sup_{X_{\rm ext}}\varphi_{X/Y}^{(1)}< \infty$. 
We refer to \cite{BPDuke} for a complete argument, based on the classical version of the Ohsawa-Takegoshi extension theorem.
\smallskip

(i)
We first show that \emph{the weights $\varphi_{X/Y}^{(1)}$ of metric $B_{1, X/Y}^{-1}|_{X_{\rm ext}}$ are upper semi-continuous}. 
This claim will be a consequence of the fact that $B_{1, X/Y}^{-1}|_{X_{\rm ext}}$ can be written as increasing limit of continuous metrics on the bundle $(K_{X/Y}+ L)|_{X_{\rm ext}}$. 
Let $(h_{L, \ep}=e^{-\varphi_{L,\ep}})_{\ep > 0}$ be an increasing sequence of smooth metrics on $L$, converging pointwise and in $L^1$ to the metric $h_L$. We recall that this can be constructed by a global convolution kernel, cf. \cite{Dnote}. 
Actually, the negative part of the curvature form associated to $(L, h_{L, \ep})$ is uniformly bounded and it converges almost everywhere to zero, but we will not use this important information here.

Let $B_{1,\ep}^{-1}|_{X_{\rm ext}}$ be the metric on $(K_{X/Y}+ L)|_{X_{\rm ext}}$ whose local weights $\varphi_{\alpha, \ep}$ are defined as follows (\cite[pp.\,346--347]{BPDuke}). 
Given a point $x_0\in \Omega_\alpha\subset X_{\rm ext}$ in a coordinate set $\Omega_\alpha$ we consider the expression
$$
	B_{1,\ep}(x_{0}) = \exp\big(\varphi_{\alpha, \ep}(x_0)\big)
	:= \sup_u \frac{\vert u^\prime_\alpha(x_0)|^2}{\Vert u\Vert^2_{y_0, \ep }},
$$
where the notations are as follows: $y_0:= f(x_0)$ is the projection of $x_0$, $u$ is a global section of the bundle $\displaystyle K_{X_{y_0}}+ L_{y_{0}}$, and $\Vert \cdot \Vert_{\, \cdot, \ep}$ means that we are using the metric $h_{L, \ep}|_{X_{y_0}}$ in defining the $L^2$ norm of $u$. 
We fix $z=(z_{1},\ldots, z_{k})$ and $t=(t_{1},\ldots,t_{\ell})$ coordinates on $\Omega_\alpha$ and $f(\Omega_\alpha)$ respectively.
The local holomorphic function $u^\prime_\alpha$ is defined as 
$u^\prime_\alpha dz = u \wedge dt$, where $dz=dz_{1}\wed\ldots\wed dz_{k}$ and $dt=dt_{1}\wed \ldots\wed dt_{\ell}$. 
In other words, $B_{1,\ep}^{-1}|_{X_{\rm ext}}$ is defined exactly as $B_{1, X/Y}^{-1}|_{X_{\rm ext}}$, except that we are using the metric $h_{L, \ep}$ for the normalization of the sections on the fibers of $f$. 

The fact that $h_{L, \ep}$ is non-singular for each $\ep> 0$, together with the choice of $y_{0} \in Y_{\rm ext}$ shows that the weights $\varphi_{\alpha, \ep}$
defined on $\Omega_\alpha$ are continuous. Given the monotonicity property of the family of metrics $h_{L, \ep}$ with respect to 
$\ep$, we infer that we have
\begin{equation} \label{monot}
\varphi_{\alpha, \ep} \geq \varphi_{\alpha, \ep^\prime}\geq \varphi_{\alpha, X/Y}^{(1)} 
\end{equation}
provided that $\ep\geq \ep^\prime$; we denote by $\varphi_{\alpha, X/Y}^{(1)}$ the local weight of the metric
$B_{1, X/Y}^{-1}|_{X_{\rm ext}}$ on $\Omega_\alpha$ with respect to the chosen trivializations. 

We show next that we have
\begin{equation}\label{limit}
\lim_{\ep\to 0}\varphi_{\alpha, \ep}= \varphi_{\alpha, X/Y}^{(1)}.
\end{equation}
To this end, we will analyze the following two cases.

\noindent $\bullet$ 
We have $\varphi_{\alpha, X/Y}^{(1)}(x_0)= -\infty$ (indeed this may happen, e.g.\ if the point $y_0$ belongs to the set $Y_{\rm ext} \setminus Y_h$). 
Let $u_\ep$ be a holomorphic section of the bundle $\displaystyle K_{X_{y_0}}+ L_{y_0}$, such that the $L^2$ norm 
$\displaystyle \Vert u_\ep\Vert^2_{y_{0}, \ep}:= \int_{X_{y_0}}|u_\ep|^2e^{-\varphi_{L, \ep}}$ is equal to 1, and such that
$$
	\varphi_{\alpha, \ep}(x_0)= \log \vert u^\prime_{\ep, \alpha}(x_0)|^2.
$$
We claim that $\lim_{\ep\to 0}\vert u^\prime_{\ep, \alpha}(x_0)|^2= 0$. 
If this is not true, then let $u_\infty$ be the limit of $u_\ep$; it is not identically zero, given our assumption. Let $\displaystyle e^{m_\ep}$ be the maximum of the function $\displaystyle (\wtil h/h_{L, \ep})|_{X_{y_0}}$, where $\wtil h$ is a smooth, fixed metric on $L$. 
If $m_\ep\to -\infty$, then we obtain a contradiction as follows. 
We have
$$
	\int_{X_{y_0}}|u_\ep|^2e^{-\varphi_{L, \ep}+ m_\ep} = e^{m_\ep}
$$
so that by the definition of $m_\ep$, we infer
$$
	\int_{X_{y_0}}|u_\ep|^2e^{-\phi_{L}} \leq e^{m_\ep},
$$
where $e^{-\phi_L}$ denotes the smooth metric $\wtil h$. 
This cannot happen if $\ep$ is small enough, because $(u_\ep)$ is converging towards a non-identically zero limit. 

Thus, we can assume that $\displaystyle \limsup_{\ep\to 0}m_\ep:= m_\infty > -\infty$; in other words, $y_{0} \in Y_h$. 
Then we have
\begin{equation}\label{section}
	\int_{X_{y_0}}|u_\infty|^2e^{-\varphi_{L}}\leq 1
\end{equation}
by a standard argument (we apply the dominate convergence theorem in the complement of the support of the multiplier ideal sheaf corresponding to $h_L|_{X_{y_0}}$), and this contradicts the fact that $\varphi_{\alpha, X/Y}^{(1)}(x_0)= -\infty$.
\smallskip

\noindent $\bullet$ 
If $\varphi_{\alpha, X/Y}^{(1)}(x_0)> -\infty$, then we know a-priori that $y_{0}\in Y_h$, and by the arguments already used in the preceding case we have $u_\infty:= \lim_\ep u_\ep$, such that (\ref{section}) holds. We easily see that 
$$\varphi_{\alpha, X/Y}^{(1)}(x_0)= \log \vert u^\prime_{\infty, \alpha}(x_0)|^2,$$
as if not, the inequality (\ref{monot}) will be contradicted. 

\noindent This concludes the proof of the upper semi-continuity of $\varphi_{X/Y}^{(1)}$.  
\smallskip

(ii)
Next we show that $\varphi_{X/Y}^{(1)}$ verifies the mean value inequality; at this point we follow the approach outlined by H. Tsuji in \cite{Ts00}, which appears in the article by Guan-Zhou \cite[\S 3.5]{GZ}. 

A first observation is that it is enough to treat the case $\dim Y= 1$. 
Given a smooth projective family of varieties $f: X\to \BD$ over the disc of radius $1$ in $\BC$ (so that $y_{0}$ will correspond to 0), the minimal norm of the extension of twisted canonical forms can be controlled in an optimal way, 
thanks to the important version of Ohsawa-Takegoshi theorem, recently obtained by B\l ocki \cite{blocki}  and \cite{GZ}. 
The result we need states as follows.

\begin{thm} \cite[Thm 2.2]{GZ}\label{xyzhu}
 Given a section $u$ of the bundle
 $\displaystyle (K_{X_{y_0}}+ L_{y_0})\otimes \CI(h_L|_{X_{y_0}})$ 
 and a positive real number $r<1$, there exists a section $U_r$ of 
 $K_{X}+ L$ extending $u$, 
 defined over the pre-image  $f^{-1}(\BD_r)$ of the disk of radius $r$ such that 
\begin{equation}\label{minext}
	\frac{1}{\pi r^2} \int_{t\in \BD_r} 
			\Big\Vert \frac{U_r}{dt}\Big\Vert _{ t}^{2} d\lambda(t) 
		\leq \Vert u\Vert _{y_0}^{2}.
\end{equation}
\end{thm}
\medskip
\noindent Actually, this result is established in \cite{GZ} only in the case of a smooth metric $h_L$; the singular version we need here can be derived from it by the usual regularization procedure.
\smallskip

We consider a particular section $u\in H^0\big(X_{y_0}, K_{X_{y_0}}+L_{y_0}\big)$ so that we have
$$
	B_{1, X/Y}(x_0)=\exp \big(\varphi_{X/Y}^{(1)}(x_0)\big)
	= \frac{\vert u^\prime(x_0)|^2}{\Vert u\Vert _{y_0}}.
$$
We extend $u$ as in Theorem \ref{xyzhu}, and thus we obtain the section $U_r$ defined on $f^{-1}(\BD_r)$ whose $L^{2}$ norm is satisfying the inequality (\ref{minext}).
Let $t\to \gamma(t)$ be a local section of $f$, such that $x_0= \gamma(0)$.
We infer that we have
$$
\varphi^{(1)}_{X/Y}(x_0)
\leq \frac{1}{\pi r^2}
	\int_{t\in \BD_r}\varphi^{(1)}_{X/Y}\big(\gamma(t)\big)d\lambda(t)
$$
thanks to the estimate (\ref{minext}), combined with the concavity of $\log$ and the fact that the absolute value of the holomorphic functions 
is log-psh.

\medskip

\noindent In conclusion, we have showed that the weights of the metric $B_{1, X/Y}^{-1}|_{X_{\rm ext}}$ are psh functions. 
\end{proof}


\subsection{Log-plurisubharmonicity of the $m$-Bergman kernel metric for $m\geq 2$} \label{Sm>1}
 
We still keep the notation $f:X\to Y$ and $(L, h_L=e^{-\varphi_{L}})$.
We will analyze here some of the properties of the metric $(B_{m,X/Y})^{-1}$ on the bundle 
$$
	mK_{X/Y}+ L
$$
whose fiberwise construction was recalled in \ref{NSc}. 
As in the case $m=1$ treated in the previous subsection, we first consider the following non-empty Zariski open subset of $Y$
$$
Y_{m,\rm ext} := \{y \in Y_{0};\ 
		h^0(X_y,mK_{X_y}+L_y) \text{ equals to the rank of } f_{\star}(mK_{X/Y}+L) \},
$$
where $Y_0\subset Y$ is the set of regular values of $f$.
We denote by $Y_{h}=\{y\in Y_{0};\ h_{y}\not\equiv +\infty\}$.

Let $\varphi$ be a psh function on an open set in $\BC^n$. The following ideal sheaf $\CJ_m(\varphi)$ will appear in our considerations
$$
\CJ_{m, x}(\varphi)
:= \big\{ f\in \CO_{\BC^n, x} ;\
	\int_{(\BC^n, x)}\vert f\vert^{2\over m}e^{-\varphi}d\lambda <\infty \big\},
$$
where $d\lam$ is the Euclidean volume form.
Although we will not use it here, we mention the result of \cite[3.4]{JCao}, showing that there exists a psh function with analytic singularities $\psi$ such that we have $\CJ_m(\varphi)= \CJ_m(\psi)$. In particular, the ideal $\CJ_m(\varphi)$ is coherent. 
Moreover, as in the case of the usual multiplier ideal sheaf, the notion above can be formulated in the case of a line bundle endowed with a metric with semi-positive curvature current. 

\noindent 
We extend the definition of $(B_{m, X/Y})^{-1}$ to the Zariski open set 
$$
	X_{m,\rm ext}= f^{-1}(Y_{m,\rm ext})
$$ 
by assigning the value $+\infty$ on the set $f^{-1}(Y_{m, {\rm ext}}\setminus Y_h)$; formally, we have the following.

\begin{dfn} \label{mBo}
{\it Let $b_{m}$ be a smooth Hermitian metric on $-(mK_{X/Y}+L)$.
For every $y \in Y_{m,\rm ext}$, we can write as $B_{m,y} = (b_{m}|_{X_{y}}) \exp (\vph_{y}^{(m)})$, where either $\vph_{y}^{(m)} \in L^{1}_{loc}(X_{y},\BR)$ or $\vph_{y}^{(m)}\equiv -\infty$ on $X_{y}$ (we are using here the notations in {\rm \ref{NSc}}).
The latter holds if and only if $V_{m,y}=H^0(X_y,(mK_{X_y}+L_{y}) \ot \CJ_m(h_y^{1/m}))= 0$ (including the case $y \in Y_{m,\rm ext} \sm Y_{h}$ by convention).
We then define a function 
$$
	\vph_{X/Y}^{(m)o} : X\to [-\infty,+\infty) \ \text{ by }  \ 
	\vph_{X/Y}^{(m)o}(x) = \vph_{y}^{(m)}(x) \text{ if } x \in X_{y}, 
$$
and define
$$
 	B_{m,X/Y}^{o} =b_{m}\exp(\vph_{X/Y}^{(m)o}).
$$
We remark that the metric $B_{m,X/Y}^{o}$ is so far only defined on $X_{m,\rm ext}\subset X$.}
\end{dfn}

By definition, $B_{m,X/Y}^{o}\not\equiv 0$ (non-trivial) if and only if there exists $y \in Y_{h} \cap Y_{m,\rm ext}$ such that $H^0(X_y,(mK_{X_y}+L_{y}) \ot \CJ_m(h_y^{1/m})) \ne 0$.
In the expression above, we have used the identification $K_{X/Y}|_{X_{y}} \cong K_{X_{y}}$ so that 
$b_{m}|_{X_{y}}$ becomes a smooth Hermitian metric on $-(mK_{X_{y}}+L_{y})$.
It is not difficult to see $B_{m,X/Y}^o$ is independent of the choices of these identifications, even though $\vph_{X/Y}^{(m)o}$ depends on them formally.
The choice of a reference metric $b_{m}$ is also not essential.
We may regard $b_{m}$ as a part of local trivializations, and then we may simply write as $B_{m,X/Y}^o =\exp(\vph_{X/Y}^{(m)o})$ and $B_{m,X/Y}^o|_{X_{y}} =B_{m,y}$ for every $y \in Y_{m,\rm ext}$.
\medskip

\noindent We recall a result in \cite{BPDuke} 
 (compare with \cite{Ts05} and the references therein for related results).

\begin{thm}\label{bp42}
\cite[4.2]{BPDuke} 
We assume that the map $f:X\to Y$ is smooth, $Y_{m,\rm ext}=Y$, and that there exists a point $y\in Y_{h}$ such that 
$$
	H^0(X_y,(mK_{X_y}+L_{y}) \ot \CJ_m(h_y^{1/m})) \ne 0.
$$
Then the dual of the relative $m$-Bergman kernel metric:\ $(B_{m, X/Y}^o)^{-1}$ in {\rm \ref{mBo}}  defines a singular Hermitian metric on $mK_{X/Y}+L$ whose associated curvature current is semi-positive.
Moreover, the local weights $\vph_{X/Y}^{(m)o}$ of this metric are upper semi-continuous.
\end{thm}

\begin{rem} 
We take this opportunity to point out two {\it imprecisions} in the formulation \cite[4.2]{BPDuke}. 
In the first place, it must be of course assumed that the metric $(B_{m, X/Y}^o)^{-1}$ (called the relative NS-metric in \cite{BPDuke}) is not identically $+\infty$. 
Secondly, the assumption that $f_\star(mK_{X/Y}+ L)$ is locally free should be replaced by $Y_{m,\rm ext}=Y$.
Just before \cite[4.2]{BPDuke}, it is stated that these two conditions are equivalent, but this is not the case. 
In fact, the property of being locally free concerns the direct image $f_\star(mK_{X/Y}+ L)$ as a sheaf: the stalks of this sheaf consists of sections over the fibers of $f$ that {\it do extend} to neighboring fibers, and local freeness means that this space of extendable sections has everywhere the same rank. 
We remark that the direct image $f_\star(mK_{X/Y}+ L)$ is torsion free, hence locally free if the dimension of the base is equal to one. 
For the proof of \cite[4.2]{BPDuke} we need the stronger property $Y_{m,\rm ext}=Y$, saying that all sections extend locally to neighboring fibers.

This distinction between local freeness and the condition $Y_{m,\rm ext}=Y$ is precisely the heart of the question of invariance of plurigenera for smooth projective families over one dimensional base, where the local freeness is automatic, and $Y_{m,\rm ext}=Y$ is exactly what is to be proved. 
\qed
\end{rem}
\medskip

\begin{rem}
We highlight here some of the new difficulties in the proof of \ref{bp42} (along the line given in \ref{m=1} for $m=1$) specific to the case $m\geq 2$; as for the complete argument, we refer to \cite[4.2]{BPDuke}. 

In the first place, the upper semi-continuity of the metric $(B_{m, X/Y}^o)^{-1}$ follows {\it mutatis mutandis} from the arguments in \S 4.1, by using a monotone approximation of the metric $h_L$ with non-singular metrics. Hence we provide no further details about this point here. 

In order to check the mean value inequality, we cannot argue as in the case $m=1$, i.e.\ by using an extension with minimal $L^{2/m}$ norm. 
The reason is very simple: it may happen that some section $u$ of the sheaf 
$(mK_{X_{y}}+L_y)\ot \CJ_m(h_y^{1/m})$ does not admits any extension locally near $X_y$ which moreover is $L^{2/m}$-integrable with respect to the metric $h_L^{1/m}$. 
We would have such an extension e.g.\ if the variety of zeros corresponding to the multiplier ideal of $h_L^{1/m}$ is $f$-vertical, but we do not want/need to impose this additional requirement. 
Also, we remark that even a-posteriori the existence of the metric $(B_{m, X/Y}^o)^{-1}$ \emph{does not} implies the existence of an extension belonging to the ideal $\CJ_m(h_L^{1/m})$, despite of the fact that we know (by the choice of the set $Y_{m,\rm ext}$) that some extension exists.
\qed

\end{rem}
\medskip
 
We now show that the metric $(B_{m,X/Y}^{o})^{-1}$ extends across the set $X\sm X_{m,\rm ext}$;
it is our main result in this section. Prior to this, we give a proof of Proposition \ref{mOT1}, as it plays an important role in our arguments.

\begin{no}
{\it Proof of  {\rm \ref{mOT1}}:\ 
The $L^{2/m}$ version of the Ohsawa-Takegoshi theorem.}\label{SOT}
We recall briefly the set-up in Proposition \ref{mOT1}. 
Let $\Omega\subset \BC^n$ be a ball of radius $r$ and let $\sg:\Omega\to \BC$ be a holomorphic function, such that $\sup_\Omega|\sg|\leq 1$; moreover, we assume that the gradient $\partial \sg$ of $\sg$ is nowhere zero on the set $V:= \{\sg=0\}$. 
We denote by $\varphi$ a plurisubharmonic function on $\Om$, such that its restriction to $V$ is well-defined (i.e., $\varphi_{|V}\not\equiv -\infty$). 

Let $u$ be a  holomorphic function on $V$ with the property that 
$$
	\int_{V} |u|^{2/m}\exp (-\varphi){{d\lambda_V}\over {|\partial \sg|^2}}= 1.
$$ 
The question is to find an extension of $u$ to $\Omega$, whose $L^{2/m}$ norm is 
bounded by the quantity above, modulo a universal constant. 
We may suppose $m>1$.

We begin with some reductions. In the first place we can assume that the function $\varphi$ is smooth, and that the functions $\sg$ (respectively $u$) can be extended in a neighbourhood of $\Omega$ (of $V$ inside $V \cap \ol\Omega$, respectively). 
Once the result is established under these additional assumptions, the general case follows by approximations and standard normal families arguments.

We can then clearly find some holomorphic $F_1$ in $\Omega$ that extends $u$ and satisfies
$$
\int_\Omega |F_1|^{2/m}\exp(-\varphi)d\lambda\leq A<\infty.
$$
We then apply the Ohsawa-Takegoshi theorem with weight
$$
\varphi_1=\varphi +(1-1/m)\log|F_1|^2
$$
and obtain a new extension $F_2$ of $f$ satisfying
$$
\int_{\Om} {{|F_2|^2}\over{|F_1|^{2-2/m}}}\exp(-\varphi)\, d\lambda
\leq C_0\int_{V} {{|u|^2}\over{|F_1|^{2-2/m}}}\exp(-\varphi) \,{{d\lambda_V}\over{|\partial \sg|^2}}
=C_0.
$$
H\"older's inequality gives that
$$
\int_\Omega |F_2|^{2/m}\exp(-\varphi)\, d\lambda \int_\Omega {{|F_2|^{2/m}}\over{|F_1|^{(2-2/m)/m}}}|F_1|^{(2-2/m)/m}\exp(-\varphi) \,d\lambda\leq 
$$
$$
\leq \left(\int_\Omega {{|F_2|^2}\over{|F_1|^{2-2/m}}}\exp(-\varphi) d\lambda \right)^{1/m}
\left(\int_\Omega|F_1|^{2/m}\exp(-\varphi)d\lambda \right)^{(m-1)/m}
$$
which is smaller than
$$
C_0^{1/m} A^{(m-1)/m}= A(C_0/A)^{1/m}=:A_1.
$$
If $A>C_0$, then $A_1<A$. We can then repeat the same argument with $F_1$ replaced by $F_2$, etc, and get a decreasing sequence of constants $A_k$, such that 
$$A_{k+1}= A_k(C_0/A_k)^{1/m}$$
for $k\geq 0$. It is easy to see that $A_k$ tends to $C_0$. Indeed, if $A_k >r C_0$ for some $r>1$, then $A_k$ would tend to zero by the relation above. This completes the proof.
\qed
\end{no}

\noindent 
For further use, we reformulate Proposition \ref{mOT1} in terms of differential forms. 
Since this is absolutely immediate, we provide no further explanations.
As in \ref{NSc}(1), we will denote by $|u|^{2\over m}=(cu\wed\ol u)^{1/m}$ for $u\in H^{0}(V, mK_{V})$,  with an absolute constant $c$.

\begin{lem} \label{mOT}
Let $m$ be a positive integer, and
let $u$  a holomorphic $m$-canonical form on $V$ such that 
$$
	\int_{V} |u|^{2\over m}e^{-\varphi}<\infty.
$$
Then there exists a holomorphic $m$-canonical form $U$ on $\Omega$ such that:
\item{(i)} $U= u\wed (d\sg)^{\otimes m}$ on $V$ i.e. the form $U$ is an extension of $u$.
\smallskip
\item{(ii)} The next $L^{2/m}$ bound holds true
$$\int_{\Omega}  |U|^{2\over m}e^{-\varphi} \leq C_0\int_{V}  |u|^{2\over m}e^{-\varphi}, $$
where $C_0$ is the constant in the Ohsawa-Takegoshi extension theorem. 
\end{lem}
\medskip

\noindent After this preparation, we can complete the proof of the following result.


\begin{thm}\label{bp43}
We assume that there exists a point $y\in Y_h\cap Y_{m,\rm ext}$ such that 
$$
	H^0(X_y, (mK_{X_{y}}+L_y)\ot \CJ_m(h_y^{1/m}))\ne 0.
$$
Then the metric $(B_{m, X/Y}^{o})^{-1}$ on $\displaystyle (mK_{X/Y}+L)|_{f^{-1}(Y_{m,\rm ext})}$ obtained in {\rm \ref{bp42}} is not identically $+\infty$, and it
extends as a singular Hermitian metric $(B_{m, X/Y})^{-1}$ on $mK_{X/Y}+L$ with semi-positive curvature current.  
\end{thm}

\begin{rem}
We ignore whether the equality $B_{m,X/Y}|_{X_{y}}= B_{m,y}$ holds for $y\in Y_{0}\sm Y_{m,\rm ext}$.
This issue will be discussed later in \ref{A26}.
\end{rem}

\begin{proof}
We recall 
the expression of $B_{m,X/Y}^o$ on a fiber of a point $y \in Y_{m,\rm ext}$.
Let $x\in X_y$ be an arbitrary point, and denote by $z=(z_{1},\ldots,z_{k})$ respectively $t=(t_{1},\ldots,t_{\ell})$ some local coordinates centered at $x$ on an open set $\Om$, respectively $y$, and assume $L$ is also trivialized on $\Om$.
We consider  a section $u\in H^0(X_y, mK_{X_y}+L_{y})$.
Then we have $\wtil u = u \wed (f^* dt)^{\ot m} \in H^0(X_y, (mK_{X}+ L)_{|X_{y}})$ via the standard identification. 
We write as $\wtil u = u' (dz)^{\ot m}$ with $u' \in H^{0}(\Om_{y}, \CO_{X_y})$, where $\Om_{y}=\Om \cap X_{y}$, and then we have
$$
	B_{m,X/Y}^o(x)=\exp \big(\varphi^{(m)o}_{X/Y}(x)\big)
	= \sup _{\Vert u\Vert _{m, y}\leq 1}\vert u^\prime(x)|^2, 
$$
where $\|u\|_{m,y} =(\int_{X_y} |u|^{2/m} h_y^{1/m})^{m/2}$. Thus, our task is to show that the above quantity is bounded from above on
the open set $X_{m,\rm ext}\subset X$. Let $u$ be a
global section of the bundle $\displaystyle mK_{X_y}+ L_{y}$, normalized by the condition
$$
	\int_{X_y} |u|^{2\over m}e^{-{1\over m}\varphi_{L}}\leq 1.
$$
By Lemma \ref{mOT} we obtain a local $m$-canonical form 
$\wtil U$ on $\Omega$ whose restriction to $\Omega_y$ is equal to $\wtil u$, and such that the next inequality holds
$$
\int_\Omega |\wtil U|^{2\over m}e^{-{1\over m}\varphi_{L}}
\leq C_0 \int_{\Omega_y}|u|^{2\over m}e^{-{1\over m}\varphi_{L}}
\le C_0,
$$
where $C_0$ is an absolute constant. 
We write as $\wtil U=U' (dz)^{\ot m}$ with $U'\in H^{0}(\Om, \CO_{X})$.
Let $\Om'\subset \Om$ be the polydisc of poly-radius $(r, \ldots, r)$ centered at $x$.
Then by the mean value inequality, we have
$$
	|U'(x)|^{2\over m} 
	\le \frac1{(\pi r^2)^{\dim X}} \int_{\Om'} |U'|^{2 \over m} d\lambda_z
	\le \frac1{(\pi r^2)^{\dim X}}\ e^{\frac1m \sup_{\Om'} \vph_L} \int_{\Om'} |\wtil U|^{2\over m}e^{-{1\over m}\varphi_{L}}.
$$
As $U'|_{\Om_{y}}=u'$, we thus have
$$
	|u^\prime(x)|^{2\over m}\leq C, \leqno (\sharp)
$$ 
where moreover the bound $C$ above does not depend at all on the geometry of the fiber $X_y$, but on the ambient manifold $X$. 
(A direct application of the mean value inequality for $|u'|^{2/m}$ on $\Om_{y}$ would encode the geometry of the fiber $X_{y}$.) 
In particular, the metric $(B_{m,X/Y}^o)^{-1}$ admits an extension to the whole manifold $X$ and its curvature current is semi-positive, so \ref{bp43} is proved. 
\end{proof}
\medskip
\subsection{A few consequences}\label{SAfew}

We further investigate the positivity of $mK_{X/Y}+L$ and give some variants and refinements of \ref{bp43}.
We still keep the notation $f:X\to Y$ and $(L,h_{L}=e^{-\vph_L})$.
The crucial observation is as follows.

\begin{rem}\label{A23} 
The constant $``C "$ in the relation $(\sharp)$ of the proof of \ref{bp43} above only depends on the sup of the $m^{th}$-root of the metric of the bundle $L$ and the geometry of the ambient manifold $X$ (and {\it not at all} on the geometry of the fiber $X_y$). 
As a consequence we see that if we apply the above arguments to a sequence $mK_{X/Y}+ L$ when $m$ varies, then the constant in question will be uniformly bounded with respect to $m$, provided that $(1/m)\varphi_{L}$ is bounded from above. 
\end{rem}

As a consequence of \ref{A23}, we obtain the following corollaries (see also \cite{Dpark} and \cite{Ts05} for related statements).

\begin{cor}\label{A24}
Suppose that there exists a very general point $w\in Y$ such that we have
$$H^0(X_w, (mK_{X_{w}}+L_w)\ot \CJ_m(h_w^{1/m}))\ne 0.$$

\noindent Then the bundle  $mK_{X/Y}+ L$ admits a singular Hermitian metric $h_{X/Y}^{(\infty)}=\exp(- \varphi^{(\infty)}_{X/Y})$ with semi-positive curvature current, whose restriction to any very general fiber $X_w$ of $f$ has the following (minimality) property: 
for any integer $k>0$ and any section $\displaystyle v\in H^0\big(X_w, kmK_{X_w}+ kL_{w}\big)$, we have
$$
	|v|^{2\over km}e^{-\frac1{m} \varphi^{(\infty)}_{X/Y}}
	\leq	\int_{X_w} |v|^{2\over km}e^{-{1\over m}\varphi_L}
$$
up to the identification of $K_{X/Y|X_w}$ with $K_{X_w}$.
\end{cor}
  
\medskip 

\noindent In the statement of \ref{A24}, we call a point $w\in Y$ {\it very general} if it is not a critical point of $f$ and any section of the bundle $k(mK_{X/Y}+ L)_{|X_w}$ extends over the nearby fibers, for any integer $k>0$.
By the usual semi-continuity arguments we infer that the set of very general points is the complement of a countable union of Zariski closed sets of codimension at least one in $Y$. 
If $h_{L}$ has a better regularity, every regular value $y\in Y_{0}$ of $f$ can be very general.
For example, if $h$ is continuous, or more generally the multiplier ideal $\CI(h_{y}^{k})=\CO_{X_{y}}$ on $X_{y}$ with $y \in Y_{0}$ for any integer $k>0$, the invariance of plurigenera \cite{Siu} and \cite{Paun} verify this.

\smallskip

\noindent 
\begin{proof}[Proof of \rm\ref{A24}]
The  proof is an immediate consequence of the arguments of \ref{bp43}. Indeed, we consider the dual of the $km$--Bergman metric (twisted by $kL$) on the bundle
$k\big(mK_{X/Y}+ L\big)$ 
and we know that the weights of its $k^{\rm th}$-root are {\it uniformly} bounded, independently of $k$ (see \ref{A23}). We obtain the metric $h_{X/Y}^{(\infty)}=\exp(- \varphi^{(\infty)}_{X/Y})$ by the usual upper envelope construction, namely
$$
	\varphi^{(\infty)}_{X/Y}
	:= {\sup} _{k\geq 1}^\star \Big( {1\over k}\varphi^{(km)}_{X/Y}\Big).
$$
\end{proof}

\noindent Corollary \ref{A24} admits a metric version which we discuss next.

\begin{cor}\label{A25}
Suppose that $X$ is projective, and that there exists a very general point $w\in Y$ and a singular Hermitian  metric $h_{w}=e^{-\vph_{w}}$ of the bundle $mK_{X_w}+ L_{w}$ with semi-positive curvature such that an integrability condition
$$
e^{\varphi_w-\varphi_L}|_{X_{w}}\in L^{{{1+\ep}\over {m}}}
$$
holds locally at each point of $X_w$, where $\ep > 0$ is a real number.
(Here $h_w$ does not mean $h_L|_{X_w}$.)

\noindent Then the bundle $mK_{X/Y}+ L$ admits a singular Hermitian metric with semi-positive curvature, whose restriction to $X_w$ is less singular than the metric $h_{w}$.
\end{cor}

\noindent  
In the statement of \ref{A25}, we call a point $w\in Y$ {\it very general}, if it is not a critical value of $f$ and if any section of the bundle $(kmK_{X/Y}+ kL+A)_{|X_w}$ extends over the nearby fibers, for any integer $k>0$.
Here $A$ is an auxiliary ample line bundle on $X$, for example, we take
$A= 2(\dim X)H$ and $H$ is a very ample line bundle on $X$.

We remark that \ref{A25} is more {\it coherent} than \ref{A24}, in the sense that we start with a metric on a very general fiber $X_w$ of $f$, and we produce a metric defined on $X$.
Even if we give ourselves a section on a very general fiber, the object we are able to produce (via \ref{A24}) will be in general a metric, which is not necessarily algebraic (i.e., it is constructed by sections of the bundle), as one would hope or guess. 

We first write $e^{-\varphi_w}$ as limit of a sequence of algebraic metrics $e^{-\varphi_w^{(k)}}$ of $(kmK_{X/Y}+ kL+ A)_{|X_w}$,
and then we apply \ref{bp43} to construct a sequence of global metrics on 
$$
	kmK_{X/Y}+ kL+ A,
$$
whose restriction to $X_w$ is comparable with $e^{-\varphi_w^{(k)}}$.
The metric we seek is obtained by a limit process. In order to complete this program 
we have to control several constants which are involved in the arguments, 
and this is possible 
thanks to the version of the Ohsawa-Takegoshi theorem \ref{mOT1}. 
It would be very interesting to give a more direct, {\it sections-less} proof of \ref{A25}. 

\medskip

\begin{proof}[Proof of {\rm \ref{A25}}]  
The choice of the bundle $A\to X$ as above is explained by the following fact we borrow from the approximation theorem in \cite{D92}, or \cite[\S 13]{Dnote}. 

We take a smooth Hermitian metric $e^{-\vph_{A}}$ of $A$ with strictly positive curvature.
Associated to the metric $e^{-\varphi_w}$ on the bundle $mK_{X_{w}}+ L_{w}$ given by hypothesis, we consider the following space of sections
\begin{equation}\label{equa1}
	\cV_k\subset H^0(X_w, kmK_{X_w}+ kL_{w}+ A_{w})
\end{equation}
defined by $u\in \cV_k$ if and only if
\begin{equation}\label{equa2}
\Vert u\Vert_w^2 
:= \int_{X_w}|u|^2\exp(-k\varphi_w-\varphi_A)dV_\omega< \infty.
\end{equation}
We denote by $e^{-\varphi_w^{(k)}}$ the metric on the bundle $kmK_{X_w}+ kL_{w}+ A_{w}$ induced by an orthonormal basis $\{u^{(k)}_{j} \}_j$ of $\cV_k$ endowed with the scalar product corresponding to $\Vert u\Vert_w^2$. 
Then it is proved in \cite[13.21]{Dnote} that there exists a constant $C_1$ independent of $k$ such that we have 
\begin{equation}\label{equa3}
	k\varphi_w+ \varphi_A \leq \varphi_w^{(k)}+ C_1.
\end{equation} 
We remark that the ample bundle $A$ is independent of the particular point $w\in Y$ as well as on the metric $e^{-\varphi_w}$. 
This fact is very important for the following arguments.

\noindent 
The H\"older inequality shows that the $(km)^{\rm th}$-root of the sections $v\in \cV_k$ are integrable with respect to the metric 
$h_L^{1/m}$, as soon as $k$ is large enough.
Indeed we have:
\begin{equation}\label{equa4}
\int_{X_w}|v|^{2\over km}e^{-{k\varphi_L+ \varphi_A\over km}}
\leq c_k\Big(\int_{X_w}|v|^2e^{-k\varphi_w-\varphi_A}dV_\omega\Big)^{{{1}\over {km}}},
\end{equation}
where we use the following notation
$$
c_k^{{{km}\over {km-1}}}
:= \int_{X_w}e^{{{k}\over {km-1}}(\varphi_w-\varphi_L)}dV_\omega
< \infty.
$$
We remark that the last integral is indeed convergent, by the integrability hypothesis
$e^{\varphi_w-\varphi_L}|_{X_{w}} \in L^{{{1+\varepsilon}\over {m}}}$, provided that $k\gg 0$. In conclusion, we get
\begin{equation}\label{equa5}
\int_{X_w}|v|^{2\over km}e^{-{k\varphi_L+ \varphi_A\over km}}
\leq C_2
\end{equation}
for every section $v:= u^{(k)}_{j}\in \cV_k$. By definition, the expression of the constant in the relation (\ref{equa5}) is the following
$$
	C_2:= \sup_{k\gg 0} c_k< \infty,
$$ 
therefore it is independent of $k$.
By Theorem \ref{bp43}, for each $k$ large enough we can construct the dual of the $km$-Bergman metric $h^{(km)}_{X/Y}$ on the bundle
$kmK_{X/Y}+ kL+ A$ with semi-positive curvature current. 
Thanks to the relation (\ref{equa5}), its restriction to 
$X_w$ is less singular than the metric $\varphi^{(k)}_w$, in the following precise way: 
\begin{equation}\label{equa6}
\varphi^{(km)}_{X/Y}|_{X_w} \geq \varphi^{(k)}_w- km\log C_2;
\end{equation}
(we identify $K_{X/Y}$ with $K_{X_w}$, which is harmless in this context, since the point $w$ is ``far'' from the critical loci of $f$).
The inequality (\ref{equa6}) combined with (\ref{equa3}) shows furthermore that we have
\begin{equation}\label{equa7}
	{1\over k}\varphi^{(km)}_{X/Y}|_{X_w} \geq \varphi_w+ C_4,
\end{equation}
where $C_4:= -m\log C_2$.

On the other hand, by \ref{A23} we have an a-priori upper bound 
\begin{equation}\label{equa8}
	{1\over k}\varphi^{(km)}_{X/Y}\leq C_5,
\end{equation}
where $C_5$ is {\it uniform with respect to $k$}. 
Then the metric we seek is obtained by the usual upper envelope construction, namely
\begin{equation}\label{equa9}
\varphi^{(\infty)}_{X/Y}
: = {\limsup_{k \to \infty}}^\star \Big({1\over k}\varphi^{(km)}_{X/Y}\Big).
\end{equation}
The local weights $\big(\varphi^{(\infty)}_{X/Y}\big)$ glue together to give a metric
for the bundle $mK_{X/Y}+ L$, since the auxiliary bundle $A$ used in the approximation process
is removed by the normalization factor $1/k$.
Moreover, the relation (\ref{equa7}) shows that we have 
\begin{equation}\label{equa10}
\varphi^{(\infty)}_{X/Y}|_{X_w} \geq \varphi_w+ \CO(1); 
\end{equation}
hence, the local weight $\varphi^{(\infty)}_{X/Y}$ restricted to the fiber $X_w$ is 
less singular than the metric $e^{-\varphi_w}$ and \ref{A25} is proved.
\end{proof}
\medskip

\begin{rem}
The Corollary \ref{A25} can be immediately adapted to the case where $X$ is only assumed to be quasi-projective. 
Another by-product of the proof of this result is the case where $f:X\to Y$ is a projective family over a Stein manifold. 
The verification of these statements is left to the interested reader. 
\end{rem}

\medskip

\begin{rem}\label{A26}
It is certainly worthwhile to understand the behavior of the metric constructed in \ref{bp43} over the {\it exceptional fibers}, where it is just defined as the (unique) extension of $(B_{m,X/Y}^{o})^{-1}$ (defined over $X_{m,\rm ext}=f^{-1}(Y_{m,\rm ext})$) as a singular Hermitian metric of $mK_{X/Y}+L$.

Let us here look at the fibers over a point $y\in Y$ where the a-priori extension of the sections is not known to be satisfied, but such that $y$ is still a regular value of $f$ (i.e., $y\in Y_{0} \sm Y_{m,\rm ext}$). 
On the bundle
$\displaystyle mK_{X_y}+ L_y$ we can consider (at least) two natural extremal metrics:
the one induced by {\it all} its global sections satisfying the $L^{2/m}$ integrability 
condition with respect to $h_L$ which we denote by $h_1$, and metric corresponding to the subspace of sections which extend locally near $y$, denoted by $h_2$.
It is clear that we have $h_1 \preccurlyeq h_2$ ($h_1$ is less singular than $h_2$, refer \cite[6.3]{Dnote}), as one can see by the comparison between the corresponding $m$-Bergman kernels. 
Our claim is

\noindent \underline{Claim}. 
{\it The metric $h^{(m)}_{X/Y}=B_{m,X/Y}^{-1}$ extends over $X_y$ to a metric which is {\rm less singular than} $h_2$ i.e., $h^{(m)}_{X/Y}|_{X_{y}} \preccurlyeq h_2$, at least if the singularities of the weight $\varphi_{L}$ are ``mild'' enough.}

Indeed, as a consequence of \ref{bp43} we have
$$
	\varphi_{X/Y}^{(m)}(x)
	= \limsup_{x^\prime\to x}\varphi_{X/Y}^{(m)}(x^\prime),
$$
where $x^\prime\in X_{m,\rm ext}$ in the limit above. 
In order to establish the result we claim, let us consider a section $u\in H^0(X_y, mK_{X_y}+ L_y)$ which computes the local weight of $h_2$ at $x$, i.e.
$\Vert u\Vert_{m,y}= 1$ and 
$$
	\sup_{\Vert v\Vert_{m,y}= 1}|v^\prime(x)|^2= |u^\prime(x)|^2
$$
as in the proof of \ref{bp43} (\cite[pp.\,346--347]{BPDuke}). 
If $x^\prime\in X_0$ is close enough to $x$ and $y^\prime:= f(x^\prime)\in Y_{m,\rm ext}$, then we certainly have
$$
	\sup_{\Vert v\Vert_{m, y^\prime}=1}|v^\prime(x^\prime)|^2
	\geq  {{|u^\prime(x^\prime)|^2}\over {\Vert u\Vert_{m, y^\prime}^2}}
$$
(here we use the fact that the section $u$ extends over the fibers near $X_y$). 
Hence in order to establish the claim above, all that we need is the relation
$$
	\liminf_{y^\prime\to y}\Vert u\Vert_{m, y^\prime}\leq \Vert u\Vert_{m,y}.
$$
This last inequality clearly holds at least if $\varphi_L$ is continuous, and we actually believe that {\it it always holds}.
\qed\end{rem}


\newpage

\section{Positivity of direct image sheaves of pluricanonical type}\label{Sdrc}

Our aim here is to show \ref{Grp}; the positivity of direct image sheaves of pluricanonical type:\ $f_{\star}(mK_{X/Y}+L)$.

\subsection{Positivity of direct image sheaves of pluricanonical type}

We shall discuss in the setup \ref{relative} throughout this subsection.
We further introduce the following subset of $Y$ to make statements clearer.
For every integer $m>0$, we set
$$
	Y_{m,h} :=\{y\in Y_{h};\  \CI(h_{y}^{1/m})=\CO_{X_y}\}. 
$$
%
%
We can see $Y_{m,h}\subset Y_{\ell,h}$ if $m<\ell$ in general.
Let us see further

\begin{lem}\label{Ymh}
{\rm (1)} If $Y_{m,h}$ is not empty, $Y_{0}\sm Y_{m,h}$ has measure zero. 

\noindent
{\rm (2)} $Y_{1,h}=\{y\in Y_{h};\  \CI(h_{y})=\CO_{X_y}\} \subset (Y_{1,h,\rm ext} \cap Y_{m,\rm ext})$ for every integer $m>0$.

\noindent
{\rm (3)} For $y \in Y_{h}$, $y \in Y_{m,h}$ if and only if $\CJ_{m}(h_{y}^{1/m}) =\CO_{X_{y}}$.
\end{lem}

\begin{proof}
(1) By Ohsawa-Takegoshi, we have $\CI(h_{y}^{1/m}) \subset \CI(h^{1/m})\cdot \CO_{X_{y}}$ for every $y \in Y_{0}$.
If $Y_{m,h}$ is not empty, then $f(V\CI(h^{1/m})) \subset Y$ is a proper analytic subset, where $V\CI(h^{1/m})\subset X$ is the complex subspace defined by $\CI(h^{1/m})$.
Thus by Fubini, $h_{y}^{1/m}\in L^{1}_{loc}$ on $X_{y}$ for a.a.\,$y \in Y_{0} \sm f(V\CI(h^{1/m}))$.

(2) 
It is clear that $Y_{1,h}\subset Y_{1,h,\rm ext}$.
Moreover by Siu \cite{Siu} (see also \cite{Paun}), if $y \in Y_{1,h}$, the natural homomorphism $f_{\star}(mK_{X/Y}+L) \to H^{0}(X_{y},mK_{X_{y}}+L_{y})$ is surjective, and hence $Y_{1,h}\subset Y_{m,\rm ext}$.

(3)
We see $\CJ_{m}(h_{y}^{1/m}) \subset \CI(h_{y}^{1/m})$, as $|u|^{2}h_{y}^{1/m} \le C_{u} |u|^{2/m}h_{y}^{1/m}$ for a constant $C_{u}>0$ depends only on any given local holomorphic function $u$.
In particular $\CI(h_{y}^{1/m})=\CO_{X_{y}}$ if $\CJ_{m}(h_{y}^{1/m}) =\CO_{X_{y}}$.
While if $\CI(h_{y}^{1/m})=\CO_{X_{y}}$, i.e., $h_{y}^{1/m}\in L^{1}_{loc}$, then $\CJ_{m}(h_{y}^{1/m}) =\CO_{X_{y}}$.
\end{proof}

We shall convert the positivity of $mK_{X/Y}+L$ to that of $f_{\star}(mK_{X/Y}+L)$ via \ref{ext}.
As a consequence of \ref{Ymh}, as soon as $f_{\star}(mK_{X/Y}+L)$ is non-zero and $Y_{m,h}\ne\emptyset$, the assumption in \ref{bp43}:\ $H^0(X_y, (mK_{X_{y}}+L_y)\ot \CJ_m(h_y^{1/m}))\ne 0$ for some $y\in Y_h\cap Y_{m,\rm ext}$ is satisfied.
We then can apply results in the previous section.
The main result in this section is the following statement.

\begin{thm}\label{Grp}
Suppose that $Y_{m,h}=\{y\in Y_{h};\ \CI(h_{y}^{1/m})=\CO_{X_y}\}$ is not empty.
Then $f_*(mK_{X/Y}+L)$ (if it is non-zero) admits a singular Hermitian metric $g_{mNS,X/Y}$ with positive curvature and satisfying a base change property on $Y_{m,\rm ext}$, which means, $g_{mNS,X/Y,y}=g_{mNS,y}$ holds on every fiber $f_{\star}(mK_{X/Y}+L)_{y} =H^{0}(X_{y},mK_{X_{y}}+L_{y})$ at $y\in Y_{m,\rm ext}$, where $g_{mNS,y}$ is the $m$-th Narashimhan-Simha Hermitian form on $H^0(X_y,mK_{X_y}+L_y)$ with respect to $h_{y}$.
In particular, if $Y$ is projective, $f_*(mK_{X/Y}+L)$ is weakly positive at every $y \in Y_{m,\rm ext}\cap Y_{m,h}$ in the sense of Nakayama.
\end{thm}

\begin{proof}
As $Y_{m,h}\ne\emptyset$, the assumption in \ref{bp43} is satisfied, and then we can put the fiberwise $m$-th Narashimhan-Simha Hermitian form $g_{mNS,y}$ on $f_*(mK_{X/Y}+L)_{y}=H^{0}(X_{y},mK_{X_{y}}+L_{y})$ for every $y \in Y_{m,\rm ext}$.
Moreover, by \ref{NSc}(4),  it satisfies $0 < \det g_{mNS,y} <+\infty$ for every $y \in Y_{m,\rm ext}\cap Y_{m,h}$.
We apply \ref{bp43} and have a singular Hermitian metric 
$h_{m-1}=(h^{(m)})^{(m-1)/m}h^{1/m}$  on $L_{m-1}:=(m-1)K_{X/Y}+L$ with semi-positive curvature, and 
$h_{m-1}|_{X_{y}}=h_{m-1,y}:=(h^{(m)}_{y})^{(m-1)/m} h_{y}^{1/m}$ on $(m-1)K_{X_{y}}+L_{y}$ for every $y \in Y_{m,\rm ext}$.

As we remarked in \ref{NSc}, we have $H^{0}(X_{y},(K_{X_{y}}+ L_{m-1}|_{X_{y}})\ot \CI(h_{m-1,y})) = H^{0}(X_{y},K_{X_{y}}+ L_{m-1}|_{X_{y}})$ for any $y \in Y_{m,h}$, in particular $Y_{m,h}\subset Y_{1,h_{m-1},\rm ext}$ holds, where
$Y_{1,h_{m-1},\rm ext}  :=\{y\in Y_{h_{m-1}};\ 
		H^{0}(X_{y},(K_{X_{y}}+L_{m-1}|_{X_{y}})\ot \CI(h_{m-1}|_{X_{y}})) = H^{0}(X_{y},K_{X_{y}}+L_{m-1}|_{X_{y}}) \}$. 
Thus, by \ref{Y1h} for $(L_{m-1},h_{m-1})$, $f_{\star}(K_{X/Y}+ L_{m-1})$ is locally free and the natural inclusion $f_{\star}((K_{X/Y}+ L_{m-1})\ot \CI(h_{m-1})) \subset f_{\star}(K_{X/Y}+ L_{m-1})$ is isomorphic both on a Zariski open subset of $Y$ containing $Y_{1,h_{m-1},\rm ext}(\supset Y_{m,h})$.

Then by \ref{ext} for $(L_{m-1},h_{m-1})$ instead of $(L,h)$, $f_{\star}(K_{X/Y}+L_{m-1})$ is positively curved with the canonical $L^{2}$-metric with respect to $h_{m-1}$.
Here the canonical $L^{2}$-metric is nothing but the singular Hermitian metric $g_{m}$ on $f_*(mK_{X/Y}+L)$, more precisely $g_{mNS,X/Y,y}=g_{mNS,y}$ for $y \in Y_{m,\rm ext}$.

The pointwise weak positivity follows from a general result \ref{Gp imply wp2}.
\end{proof}

\subsection{Variants and refinements}\label{Sdrc2}

We will state some corollaries in the rest, and still keep the set up  \ref{relative}.
We first give a statement of MMP flavor.

\begin{cor}
Let $\Delta$ be an effective $\BQ$-divisor on $X$ such that a pair $(X,\Delta)$ is klt.
Then $f_*(\CO_X(\lfloor m(K_{X/Y}+\Delta)\rfloor))$ has a singular Hermitian metric with positive curvature.
\end{cor}

Here $\lfloor \, \bullet \, \rfloor$ is the round down.
In this algebraic setting, we can set up a singular Hermitian metric $h$ so that $h$ has the so-called algebraic/analytic singularities; $(L,h)$ corresponds to a line bundle $\CO_X(\lfloor m\Delta \rfloor)$ with a singular Hermitian metric obtained by a defining section of $\lfloor m\Delta \rfloor$.
In such a case, $Y_{h}=\{y\in Y_{0};\ X_y \not\subset \Supp \Del\}$ and $Y_{m,h} = \{y\in Y_{h};\ (X_{y}, \frac1m \lfloor m\Delta_{y} \rfloor)$ is klt$\}$ are Zariski open (where $\Delta_{y}=\Delta|_{X_{y}}$ is the restriction as a divisor), and hence $g_{mNS,X/Y,y}=g_{mNS,y}$ on some Zariski open subset.

If we pose a stronger regularity assumption on the singular Hermitian metric $h$, we can obtain a stronger positivity conclusion on the direct image sheaves.
For example,

\begin{cor}\label{mcor}
{\rm (1)} 
Suppose that $Y_{1,h}:=\{y\in Y_{h};\  \CI(h_{y})=\CO_{X_y}\}$ is not empty.
Then for every $m>0$, the $m$-th Narashimhan-Simha singular Hermitian metric $g_{m}=g_{mNS,X/Y}$ on $f_{\star}(mK_{X/Y}+L)$ in {\rm \ref{Grp}} satisfies $0<\det g_{m} <\infty$ on $Y_{1,h}$ (which is contained in $Y_{m,\rm ext} \cap Y_{m,h}$ by \ref{Ymh}).

{\rm (2)} Suppose $h$ is continuous on $f^{-1}(Y_0)\subset X$. 
Then the $m$-th Narashimhan-Simha singular Hermitian metric $g_{m}$ on  $f_{\star}(mK_{X/Y}+L)$ in {\rm \ref{Grp}} is continuous and $0<\det g_{m} <\infty$ both on $Y_{0}$.
In particular, if $Y$ is projective, $f_{\star}(mK_{X/Y}+L)$ is weakly positive at every $y\in Y_{0}$ in the sense of Nakayama. 
\end{cor}

Various weak positivities are known for direct images of pluricanonical type $f_{\star}(mK_{X/Y})$ after Viehweg.
However, according to Nakayama, Viehweg's method (and also Nakayama's method in \cite{Nbook}) can only show the weak positivity at some point $y\in Y$, and do not imply weak positivity at all points $y\in Y_{0}$.

It can be possible to draw several variants from the basic results above.
For example, we have a result with variable $m$.
The following statement (1) is a variant of \ref{A24}.

\begin{cor}\label{w/min}
Let $f:X\to Y$ be a projective surjective morphism of complex manifolds with connected fibers, and let $Y_{0}\subset Y$ be the maximum Zariski open subset where $f$ is smooth.
Suppose that the Kodaira dimension of general fibers are non-negative:\ $\kappa(X_{y}) \ge 0$.
Then

{\rm (1)} $K_{X/Y}$ is pseudo-effective and admits a singular Hermitian metric $h_{X/Y}$ with semi-positive curvature and satisfying, for any $y\in Y_{0}$, any integer $m>0$, and any $u \in H^{0}(X_{y}, mK_{X_{y}})$,
$$
	\big( |u|^{2}h_{X/Y}^{m}|_{X_{y}} \big)^{1/m} \le \int_{X_{y}} |u|^{2/m}
$$
(up to the identification of $K_{X/Y}|_{X_{y}}$ with $K_{X_{y}}$).

{\rm (2)} 
For any integer $m>0$ with $f_{\star}(mK_{X/Y}) \ne 0$, $f_{\star}(K_{X/Y}+(m-1)K_{X/Y})$ endows with a canonical $L^{2}$-metric $g_{m}'$ with respect to $h_{X/Y}^{m-1}$ (of $(m-1)K_{X/Y}$ with $h_{X/Y}$ in (1)), is positively curved, and $g_{m}'$ is continuous and $0<\det g_{m}' <\infty$ both on $Y_{0}$.
\end{cor}

When $X$ is further compact, we can use the notion (from \cite[1.5]{DPS}) of singular Hermitian metric with minimal singularities, say $h_{min}$, of $K_{X/Y}$ with semi-positive curvature, and then $h_{min} \preccurlyeq h_{X/Y}$, i.e, $h_{min}$ is less singular than $h_{X/Y}$.
Then the same conclusion in \ref{w/min} holds for $h_{min}$ instead of $h_{X/Y}$.

\begin{proof}[Proof of {\rm \ref{w/min}}]
We are going to use \ref{A24} with $(L,h_{L})=(\CO_{X},1)$.
Noting Siu's invariance of plurigenera \cite{Siu}, every point in $Y_{0}$ is a very general point in \ref{A24}.

(1)
In \ref{A24}, $m$ is fixed. We vary $m$ here.
For every integer $m$ with $f_{\star}(mK_{X/Y})\ne 0$, $mK_{X/Y}$ admits a singular Hermitian metric $h^{(m)}$ with semi-positive curvature and satisfying, for any $y\in Y_{0}$, any integer $k>0$, and any $v \in H^{0}(X_{y}, kmK_{X_{y}})$,
$$
	\big( |v|^{2} (h^{(m)}|_{X_{y}})^{k} \big)^{1/(km)} \le \int_{X_{y}} |v|^{2/(km)}
$$
holds (up to the identification of $K_{X/Y}|_{X_{y}}$ with $K_{X_{y}}$).
Here $|v|^{2}(h^{(m)}|_{X_{y}})^{k}$ is a function on $X_{y}$, and $ |v|^{2/(km)}$ is a semi-positive volume form on $X_{y}$.

We consider $h_{X/Y}:= \inf_{m}^{\star} (h^{(m)})^{1/m}$ a lower regularized limit (or writing as $h^{(m)}=e^{- \vph^{(m)}}$ on each local coordinate with a weight function, we consider $\vph:=\sup_{m}^{\star}(\frac1m \vph^{(m)})$ an upper regularized limit), where $m$ runs through all integers $m>0$ with $f_{\star}(mK_{X/Y})\ne 0$.
As $(L,h_{L})=(\CO_{X},1)$ in \ref{A23}, an upper bound of $\vph^{(m)}$ depends only on the geometry of $X$, and hence for every relatively compact set $U \subset X$, there exists a constant $C_{U}>0$ such that $\vph^{(m)} \le \frac1m C_{U}\le C_{U}$.
As a consequence, $h_{X/Y}\not\equiv+\infty$, namely $h_{X/Y}$ defines a singular Hermitian metric, and then by definition, $h_{X/Y}$ is less singular than $(h^{(m)})^{1/m}$ for any integer $m>0$ with $f_{\star}(mK_{X/Y})\ne 0$.

(2)
We let $L=(m-1)K_{X/Y}$ with a singular Hermitian metric $h=h_{X/Y}^{m-1}$.
By (1), we have $H^{0}(X_{y},(K_{X_{y}}+L_{y})\ot \CI(h_{y}))=H^{0}(X_{y},K_{X_{y}}+L_{y})$ for any $y\in Y_{0}$, namely $Y_{1,h,\rm ext}=Y_{0}$.
In particular the natural inclusion $f_{\star}((K_{X/Y}+L)\ot \CI(h)) \subset f_{\star}(K_{X/Y}+L)$ is isomorphic over $Y_{0}$ by \ref{Y1h}.
We can further apply \ref{ext}.
\end{proof}

\newpage

\end{document}